\documentclass[14 pt, a4paper]{article}
\usepackage{latexsym}
\usepackage{floatpag}
\usepackage{multirow}
\usepackage{graphicx}
\usepackage{amsmath}
\usepackage{amssymb}
\usepackage{tikz}
\usepackage{caption}
\usepackage{tkz-euclide}
\usetikzlibrary{positioning}
\usepackage{centernot}
\usepackage{amsfonts}
\usepackage{changepage}
\usepackage{enumerate}
\usepackage[a4paper,left=2cm,right=2cm,top=2cm,bottom=2cm]{geometry}
\usepackage{diagbox}
\usepackage{authblk}
\usepackage{tabu}
\usepackage{amsthm}
\usepackage{hyperref}
\makeatletter
\def\th@plain{%
  \thm@notefont{}% same as heading font
  \itshape % body font
}
\def\th@definition{%
  \thm@notefont{}% same as heading font
  \normalfont % body font
}
\makeatother
\begin{document}
\title{$K_r$-saturated Graphs and the Two Families Theorem}
\author{Asier Calbet}
\affil{School of Mathematical Sciences \\ Queen Mary, University of London \\ Mile End Road, London E1 4NS \\ United Kingdom
\\ \href{mailto:a.calbetripodas@qmul.ac.uk}{\nolinkurl{a.calbetripodas@qmul.ac.uk}}}
\date{}
\maketitle

\newtheorem{theorem}{Theorem}
\newtheorem{lemma}{Lemma}
\newtheorem{conj}{Conjecture}
\newtheorem{q}{Question}
\newtheorem{probl}{Problem}
\theoremstyle{definition}
\newtheorem{definition}{Definition}
\newtheorem{rem}{Remark}
\theoremstyle{plain}

\begin{abstract}

Given a graph $H$, we say that a graph $G$ is $H$-saturated if $G$ contains no copy of $H$ but adding any new edge to $G$ creates a copy of $H$. Let $sat(n,K_r,t)$ be the minimum number of edges in a $K_r$-saturated graph on $n$ vertices with minimum degree at least $t$. Day showed that for fixed $r \geq 3$ and $t \geq r-2$, $sat(n,K_r,t)=tn-c(r,t)$ for large enough $n$, where $c(r,t)$ is a constant depending on $r$ and $t$, and proved the  bounds 

$$ 2^t t^{3/2} \ll_r c(r,t) \leq t^{t^{2t^2}} $$

\noindent for fixed $r$ and large $t$. In this paper we show that for fixed $r$ and large $t$, the order of magnitude of $c(r,t)$ is given by $c(r,t)=\Theta_r \left(4^t t^{-1/2} \right)$. Moreover, we investigate the dependence on $r$, obtaining the estimates

$$ \frac{4^{t-r}}{\sqrt{t-r+3}} + r^2 \ll c(r,t) \ll \frac{4^{t-r} \min{(r,\sqrt{t-r+3})}}{\sqrt{t-r+3}} + r^2 \ . $$
 
\noindent We further show that for all $r$ and $t$, there is a finite collection of graphs such that all extremal graphs are blow-ups of graphs in the collection.

Using similar ideas, we show that every large $K_r$-saturated graph with $e$ edges has a vertex cover of size $O(e / \log e)$, uniformly in $r \geq 3$. This strengthens a previous result of Pikhurko. We also provide examples for which this bound is tight.

A key ingredient in the proofs is a new version of Bollob\'as's Two Families Theorem.

\end{abstract}

\section{Introduction}\label{intro}

Given a graph $H$, we say that a graph $G$ is $H$-saturated if it is maximally $H$-free, meaning $G$ contains no copy of $H$ but adding any new edge to $G$ creates a copy of $H$. The saturation problem is to determine, or at least estimate, the saturation number $sat(n,H)$, defined (for graphs $H$ with at least one edge) to be the minimum number of edges in an $H$-saturated graph $G$ on $n$ vertices. We are usually interested in fixed $H$ and large $n$. The saturation problem is dual to the Tur\'an forbidden subgraph problem of determining the extremal number $ex(n,H)$, defined to be the maximum number of edges in an $H$-saturated graph $G$ on $n$ vertices. However, much less is known for the saturation problem than for the  Tur\'an problem. See \cite{Reference 1} for a survey of the saturation problem.  \\

For certain $H$ the saturation number is known exactly, however. In particular, Erd\H{o}s, Hajnal and Moon showed (Theorem 1 in \cite{Reference 2}) that for $r \geq 2$ and $n \geq r-2$, $sat(n,K_r)=(r-2)n-\binom{r-1}{2}$ and that the unique extremal graph consists of a $K_{r-2}$ fully connected to an independent set of size $n-(r-2)$ (for $n<r$, $K_n$ is the unique $K_r$-saturated graph on $n$ vertices). Note that for large $n$ this graph contains many vertices of degree $r-2$. Moreover, this is the smallest possible degree of a vertex in a $K_r$-saturated graph on $n \geq r-1$ vertices. One might therefore ask what happens if we forbid vertices of degree $r-2$, and more generally, vertices of small degree. This leads us to define $sat(n,K_r,t)$ for $r \geq 3$, $t \geq r-2$ and large enough $n$ to be the minimum number of edges in a $K_r$-saturated graph $G$ on $n$ vertices with $\delta(G) \geq t$ (more precisely, such graphs exist if and only if $n \geq (r-1)t / (r-2) $). This quantity (or rather the one obtained by replacing the condition $\delta(G) \geq t$ with $\delta(G)=t$) was first considered by Duffus and Hanson (second paragraph of the Introduction in \cite{Reference 3}). \\

Proving a conjecture of Bollob\'as, Day showed that for fixed $r$ and $t$, $sat(n,K_r,t)=tn-c(r,t)$ for large enough $n$, where $c(r,t)$ is some constant depending on $r$ and $t$. He raised the problem of estimating $c(r,t)$ and proved the bounds

$$ 2^t t^{3/2} \ll_r c(r,t) \leq t^{t^{2t^2}} $$

for fixed $r$ and large $t$ (Theorem 1.2 and the previous two paragraphs in \cite{Reference 4})(see section \ref{asymptotic notation} for precise definitions of asymptotic notation). In this paper we show that for fixed $r$ and large $t$, the order of magnitude of $c(r,t)$ is given by $c(r,t) \asymp_r 4^t t^{-1/2}$. Moreover, we investigate the dependence on $r$.

\begin{theorem}\label{c est}

For all integers $r \geq 3$ and $t \geq r-2$,

$$ \frac{4^{t-r}}{\sqrt{t-r+3}} + r^2 \ll c(r,t) \ll \frac{4^{t-r} \min{(r,\sqrt{t-r+3})}}{\sqrt{t-r+3}} + r^2 \ . $$

\end{theorem}

\newtheorem*{repth1}{Theorem \ref{c est}}

Note that the lower and upper bounds are the same up to the $\min{(r,\sqrt{t-r+3})}$ factor. We conjecture that the lower bound is tight. The $r^2$ terms cannot be absorbed into the other terms, because when $t$ is not much larger than $r-2$, the $r^2$ terms dominate. Indeed, for such $t$ Theorem \ref{c est} gives $c(r,t) \asymp r^2$. \\

We further characterise the extremal graphs for large $n$, and more generally, $K_r$-saturated graphs $G$ on $n$ vertices with $\delta(G) \geq t$ and $e(G)$ not much larger than $sat(n,K_r,t)$.

\begin{theorem}\label{blow-up}

Let $r \geq 3$,  $t \geq r-2$ and $k$ be integers. Then there is a finite collection $C$ of graphs such that the following holds. Let $G$ be a graph on $n$ vertices. Then $G$ is $K_r$-saturated with $\delta(G) \geq t$ and $e(G) = tn+k$ if and only if $G$ is obtained from a graph in $C$ by blowing up non-adjacent vertices of degree $t$.

\end{theorem}

\newtheorem*{repth2}{Theorem \ref{blow-up}}

(Blowing up a vertex refers to replacing it by several copies - see section \ref{bu} for the precise definition of a blow-up.) \\

Using ideas similar to those in the proof of Theorem \ref{c est} one can prove further structural results about $K_r$-saturated graphs. Pikhurko showed that for fixed $r$ and $t$, every $K_r$-saturated graph has only a bounded number of vertices of degree at most $t$ that are adjacent to some other vertex of degree at most $t$ (Theorem 8 in \cite{Reference 5}). He then used this result to deduce that for fixed $r$ and large $n$, every $K_r$-saturated graph on $n$ vertices with $O(n)$ edges has a vertex cover (a set of vertices intersecting every edge) of size $O_r(n \log \log n / \log n)$  (Lemma 9 in \cite{Reference 5}). In fact, the same argument shows more generally that for fixed $r$ and large $e$, every $K_r$-saturated graph with $e$ edges has a vertex cover of size $O_r(e \log \log e / \log e)$. \\

Let $f(r,t)$ be the maximum over all $K_r$-saturated graphs of the number of vertices of degree at most $t$ that are adjacent to some other vertex of degree at most $t$. Pikhurko's proof gives 

$$ f(r,t) \leq e^{2(r-2)t \log t + O_r(t)} $$ 

for fixed $r \geq 3$ and large $t$. In this paper we show that for fixed $r \geq 3$ and large $t$ the order of magnitude of $f(r,t)$ is given by $f(r,t) \asymp_r 4^t t^{-1/2}$. Moreover, we investigate the dependence on $r$.

\begin{theorem}\label{vert bound}

For all integers $r \geq 3$ and $t \geq r-2$,

$$ \frac{4^{t-r}}{\sqrt{t-r+3}} + r \ll f(r,t) \ll \frac{4^{t-r} \min{(r,\sqrt{t-r+3})}}{\sqrt{t-r+3}} + r \ . $$

\end{theorem}

\newtheorem*{repth3}{Theorem \ref{vert bound}}

As in Theorem \ref{c est}, the lower and upper bounds are the same up to the $\min{(r,\sqrt{t-r+3})}$ factor. We again conjecture that the lower bound is tight. As before, the $r$ terms cannot be absorbed into the other terms, because when $t$ is not much larger than $r-2$, the $r$ terms dominate. Indeed, for such $t$ Theorem \ref{vert bound} gives $f(r,t) \asymp r$.  \\

Using this improved upper bound for $f(r,t)$ we deduce in the same way as Pikhurko a strengthening of his vertex cover result.

\begin{theorem}\label{cover}

Every large $K_r$-saturated graph ($r \geq 3$) with $e$ edges has a vertex cover of size $O(e / \log e)$.

\end{theorem}

\newtheorem*{repth4}{Theorem \ref{cover}}

Note that this bound is uniform in $r$. \\

We trivially have that every graph on $n$ vertices has a vertex cover of size $n$, so it follows that every large $K_r$-saturated graph $(r \geq 3)$ with $n$ vertices and $e$ edges has a vertex cover of size $O(\min{(e / \log e,n)})$. Note that for a $K_r$-saturated graph with $n \geq r \geq 3$ vertices and $e$ edges to exist we must have

$$ rn \ll (r-2)n-\binom{r-1}{2} = sat(n,K_r) \leq e \leq ex(n,K_r) = e\left(T(n,r-1)\right) \ll n^2 \ . $$

Our next theorem states that there exist $K_r$-saturated graphs with $n$ vertices and $e$ edges for which the $O(\min{(e / \log e,n)})$ vertex cover bound is tight for all $r \geq 3$ and all such permitted orders of magnitude for $n$ and $e$.

\begin{theorem}\label{tight}

Let $r \geq 3$ be an integer and $n$ and $e$ be large quantities with $rn \ll e \ll n^2$. Then there exists a $K_r$-saturated graph $G$ with $|G| \asymp n$ and $e(G) \asymp e$ such that every vertex cover in $G$ has size $\Omega(\min{(e / \log e,n)})$.

\end{theorem}

\newtheorem*{repth5}{Theorem \ref{tight}}

The problems of estimating $c(r,t)$ and $f(r,t)$ turn out to be closely related to a celebrated result in extremal set theory known as the Two Families Theorem. This theorem was first proved by Bollob\'as and since then many different versions of the theorem have been proven (see section \ref{context} for more on the Two Families Theorem). To prove Theorems \ref{c est} and \ref{vert bound} we will need the following new version of the theorem.

\begin{theorem}\label{TF}

Let $a,b \geq c \geq 0$ be integers and $(A_i,B_i)_{i \in I}$ be a sequence of pairs of finite sets, indexed by a finite, ordered set $I$, with the following properties. 

\begin{enumerate}
    \item $|A_i \cap \cup_{k < i} \ A_k| \leq a$ and $|B_i| \leq b$ for all $i \in I$.
    \item $|A_i \cap B_i| \leq c  $ for all $i \in I$.
    \item $|A_i \cap B_j| > c $ for all $i < j \in I$.
\end{enumerate}

Then $|I| \ll \binom{a+b-2c+1}{a-c+1}$. 

\end{theorem}

As we shall see later, Theorem \ref{TF} is tight (see Theorem \ref{better estimates}). The novelty in Theorem \ref{TF} is that it only requires the bound $|A_i \cap \cup_{k < i} \ A_k| \leq a$ rather than $|A_i| \leq a$.  \\

The rest of the paper is organised as follows. We first establish some notation in section \ref{notation}. We then construct graphs in section \ref{constr} which we will use to prove the lower bounds in Theorems \ref{c est} and \ref{vert bound} and construct the graphs in Theorem \ref{tight} in the special case $r=3$. In section \ref{bu} we define blow-ups and discuss their relation with $sat(n,K_r,t)$, which leads to an alternative definition of $c(r,t)$ that we will use throughout the rest of the paper. \\

Next, we define conical vertices in section \ref{sec con} and use them to obtain inequalities for $c(r,t)$ and $f(r,t)$ which we will use to deduce the lower bounds in Theorems \ref{c est} and \ref{vert bound} for the general case $r \geq 3$ from the special case $r=3$. We will also use conical vertices when proving Theorem \ref{tight} to similarly reduce the general case $r \geq 3$ to the special case $r=3$. We conclude section \ref{sec con} with a conjecture regarding conical vertices. We then prove Theorems \ref{c est} through \ref{tight} in section \ref{mainresults}. \\

In section \ref{TFT} we first provide some background on the Two Families Theorem to put Theorem \ref{TF} into context. We then prove Theorem \ref{TF} and further related results. Next, we discuss the known values of $c(r,t)$ in section \ref{known} and complete the work of Duffus and Hanson in \cite{Reference 3} on the case $r=3=t$ by determining the extremal graphs. Finally, in section \ref{concl}, we first speculate about the exact value of $c(3,t)$ for large $t$. We then discuss how one might remove the $\min{(r,\sqrt{t-r+3})}$ factor from the upper bounds in Theorems \ref{c est} and \ref{vert bound}.

\section{Notation}\label{notation}

In this section we establish some graph theory and asymptotic notation and state some order theory definitions.

\subsection{Graph theory notation}\label{graph notation}

Given a graph $G$, we denote by $|G|$, $e(G)$, $\delta(G)$, $\chi(G)$ and $\omega(G)$ the number of vertices, number of edges, minimum degree, chromatic number and clique number of $G$, respectively. We write $V(G)$ for the vertex set of $G$. Given a subset $S \subseteq V(G)$, we denote by $G[S]$ the subgraph of $G$ induced by $S$. We let $\omega(S)=w(G[S])$ and $e(S)=e(G[S])$. Given a subset $T \subseteq V(G) \setminus S$, we write $e(S,T)$ for the number of edges between $S$ and $T$. Given a vertex $v \in V(G)$, we denote by $\Gamma(v)$ the neighbourhood of $v$ in $G$. Given an integer $s \geq 0$, we write $G^s$ for the graph obtained by adding $s$ conical vertices to $G$ (see section \ref{sec con} for the definition of a conical vertex).

\subsection{Asymptotic notation}\label{asymptotic notation}

Given two real valued functions $f$ and $g$ of several variables, we write $f \ll g$ if there exists a constant $C>0$ such that $f \leq Cg$ for all possible values of the variables. We write $f \gg g$ if $g \ll f$, and $f \asymp g$ if $f \ll g$ and $f \gg g$. $O(g)$ denotes a function $f$ such that $f \ll g$, and $\Omega(g)$ denotes a function $f$ such that $f \gg g$. Given a variable, say $r$, if instead of being a constant, $C$ is a function of $r$, we write $f \ll_r g$. In other words, $f \ll_r g$ if $f \ll g$ for any fixed $r$. Similarly, given other variables, say $t$ and $k$, we can define $f \ll_{r,t} g$, $f \asymp_r g$, $O_{r,t,k}(g)$, etc.

\subsection{Order theory definitions}\label{induced orders}

Given an order $<$ on a set $S$, the dual order $<'$ on $S$ is given by $a<'b$ if and only if $a>b$. Given an order $<$ on a set $S$ and a subset $T \subseteq S$, the induced order $<'$ on $T$ is given by $a<'b$ if and only if $a<b$. Given orders $<_i$ on disjoint sets $S_i$, indexed by a set $I$ with order $<$, the sum of the orders is the order $<'$ on $\cup_{i \in I} \ S_i$ given by $a<'b$ if and only if $a \in S_i$ and $b \in S_j$ for some $i<j \in I$ or $a,b \in S_i$ and $a<_ib$ for some $i \in I$.

\section{Construction}\label{constr}

In this section we define a graph $G_t$ for each integer $t \geq 3$. We will use these graphs to prove the lower bounds in Theorems \ref{c est} and \ref{vert bound} and construct the graphs in Theorem \ref{tight} in the special case $r=3$. Let $S$ be a set of size $2(t-1)$ and $R=\{ T \subseteq S : \ |T|=t-1 \}$. The vertex set of $G_t$ is $S \cup R$. The edges are as follows.

\begin{itemize}

\item $S$ is an independent set.
\item An element $s \in S$ is adjacent to a set $T \in R$ if and only if $s \in T$. 
\item Two sets $T,T' \in R$ are adjacent if and only if $T \cap T'= \emptyset$.

\end{itemize}

\begin{centering}
     
\begin{tikzpicture}[scale=0.5, every node/.style={scale=1.3}]
      
             \draw (0,0) ellipse (1 and 2);
             \draw (7,0) ellipse (2 and 4);
             \draw (0,3) node {$S$};
             \draw (7,5) node {$R$};
             \filldraw[black] (7,2) circle (3pt) node[anchor=south]{$T$};
             \filldraw[black] (7,-2) circle (3pt) node[anchor=north]{$S \setminus T$};
             \filldraw[black] (0,0) circle (3pt) node[anchor=east]{$s$};
             \draw (0,0) -- (7,2)  node[midway, sloped, above] {$s \in T$};
             \draw (7,-2) -- (7,2);
             \draw[dashed] (0,0) -- (7,-2);
            
\end{tikzpicture}


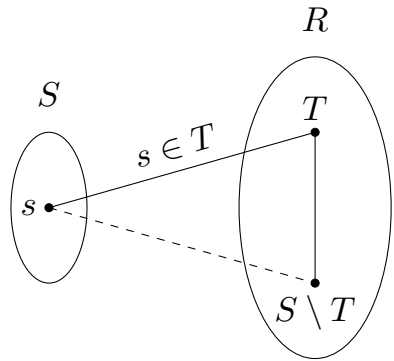
\captionof{figure}{The graph $G_t$.}

\end{centering}
   
\leavevmode\newline It is easy to check that $G_t$ has the following properties, which we will use later.

\begin{enumerate}

\item $G_t$ is $K_3$-saturated.
\item $|G_t|=\binom{2(t-1)}{t-1} + 2(t-1)$.
\item $e(G_t) = (t-1/2)\binom{2(t-1)}{t-1}$.
\item Every vertex in $S$ has degree at least $t$ and every vertex in $R$ has degree exactly $t$.  
\item $G_t[R]$ is a matching with $\binom{2(t-1)}{t-1}$ vertices.
\item $\chi(G_t)=3$.
\item Let $G_t'$ be the graph obtained from $G_t$ by adding a new vertex and joining it to all vertices in $S$. Then $G_t'$ has the property of not only being $K_3$-saturated, but also of remaining so after arbitrarily removing pairs of matching vertices from $R$.

\end{enumerate}

\section{Blow-ups}\label{bu}

In this section we first define blow-ups and observe that the operation of blowing up preserves both the properties of being $K_r$-saturated (with some trivial exceptions) and of having large minimum degree. We then explain how Day proved that for fixed $r$ and $t$, $sat(n,K_r,t)=tn-c(r,t)$ for large enough $n$, which leads to an alternative definition of $c(r,t)$. We define a similar quantity $c'(r,t)$ and pose a question about how these two quantitites are related. Finally, we define collections of graphs $C(r,t,k)$ and $C'(r,t,k)$, show that $C'(r,t,k)$ is always finite and observe that Theorem \ref{blow-up} is equivalent to the statement that $C(r,t,k)$ is always finite. \\

A blow-up of a graph $H$ is a graph $G$ obtained by replacing each vertex of $H$ with some positive number of copies. Two copies in $G$ are adjacent if and only if the vertices in $H$ they are copies of are adjacent (in particular copies of the same vertex are non-adjacent). It is easy to check that $G$ is $K_r$-saturated if and only if $H$ is $K_r$-saturated, unless $H$ is a clique on at most $r-2$ vertices and $G$ is a proper blow-up of $H$. Note that $ \delta(G) \geq \delta(H)$. \\

Hence, given integers $r \geq 3$ and $t \geq r-2$ and a $K_r$-saturated graph $H$ with $\delta(H) \geq t$, we can generate other $K_r$-saturated graphs $G$ with $\delta(G) \geq t$ by taking blow-ups of $H$ (the conditions $\delta(H) \geq t \geq r-2$ rule out the trivial exceptions). To minimise the number of edges in $G$ given the number of vertices, one should blow up non-adjacent vertices in $H$ of minimum degree. We then have $e(G)-\delta(H)|G|=e(H)-\delta(H)|H|$. In particular, if $\delta(H)=t$, $e(G)-t|G|=e(H)-t|H|$. \\

Day proved that for fixed $r$ and $t$, $sat(n,K_r,t)=tn-c(r,t)$ for large enough $n$ as follows. He first proved (and this is where the difficulty lies) that for fixed $r$ and $t$ the quantity $t|G|-e(G)$ is bounded above when $G$ ranges over all $K_r$-saturated graphs with $\delta(G) \geq t$ (Theorem 1.1 in \cite{Reference 4}). We can thus define $$ c(r,t)= \max_{\substack{G \ K_r\text{-sat.} \\ \delta(G) = t}} t|G|-e(G) \  \ \ \  \text{and} \ \ \ \ c'(r,t)= \max_{\substack{G \ K_r\text{-sat.} \\ \delta(G) \geq t}} t|G|-e(G) \ . $$

(It is easy to check that for all integers $r \geq 3$ and $t \geq r-2$, $K_r$-saturated graphs $G$ with $\delta(G)=t$ do exist.)  \\ 

Take a $K_r$-saturated graph $H$ with $\delta(H) = t$ and $t|H|-e(H)=c(r,t)$. Then by blowing up non-adjacent vertices in $H$ of degree $t$, we can construct $K_r$-saturated graphs $G$ with $\delta(G) \geq t$ (in fact with equality), $|G|=n$ and $e(G) = tn-c(r,t)$ for all $n \geq |H|$, so $sat(n,K_r,t) \leq tn-c(r,t)$ for large enough $n$. Now suppose $G$ is a $K_r$-saturated graph on $n$ vertices with $\delta(G) \geq t$. If $\delta(G)=t$, $e(G) \geq tn-c(r,t)$, whereas if $\delta(G) > t$, $e(G) \geq (t+1)n-c'(r,t+1)$. But $(t+1)n-c'(r,t+1) > tn-c(r,t)$ for large enough $n$, so $e(G) \geq tn-c(r,t)$ if $n$ is large enough.  Hence $sat(n,K_r,t) \geq tn-c(r,t)$ for large enough $n$, so $sat(n,K_r,t)=tn-c(r,t)$ for large enough $n$. \\

Clearly $c'(r,t) \geq c(r,t)$. It is unclear whether we have equality.

\begin{q}\label{q}

Do we have $c'(r,t)=c(r,t)$?

\end{q}

We now define collections of graphs $C(r,t,k)$ and $C'(r,t,k)$ for all integers $r \geq 3$, $t \geq r-2$ and $k$. Let $C(r,t,k)$ be the collection of $K_r$-saturated graphs $G$ with $\delta(G) = t$ and $e(G)=t|G|+k$ which are minimal in the sense that they cannot be obtained from a smaller $K_r$-saturated graph $H$ with $\delta(H) = t$ and $e(H)=t|H|+k$ by blowing up non-adjacent vertices in $H$ of degree $t$. Let $C'(r,t,k)$ be the collection of $K_r$-saturated graphs $G$ with $\delta(G) > t$ and $e(G)=t|G|+k$. \\

Note that for $G \in C'(r,t,k)$, $c'(r,t+1) \geq (t+1)|G|-e(G) = |G|-k$, which implies $|G| \leq c'(r,t+1) + k$. It follows that $C'(r,t,k)$ is always finite. Hence  Theorem \ref{blow-up} is equivalent to the statement that $C(r,t,k)$ is always finite. We will prove this in section \ref{mainresults}. For all integers $r \geq 3$ and $t \geq r-2$, let $C(r,t)=C(r,t,-c(r,t))$ and $C'(r,t)=C'(r,t,-c'(r,t))$. Note that for fixed $r$ and $t$ and large enough $n$, the extremal graphs for $sat(n,K_r,t)$ are precisely the graphs obtained from a graph in $C(r,t)$ by blowing up non-adjacent vertices of degree $t$.

\section{Conical vertices}\label{sec con}

In this section we first define conical vertices and note that the operation of adding a conical vertex interacts well with the property of being $K_r$-saturated. This fact allows us to obtain some inequalities for $c(r,t)$ and $f(r,t)$ which we will use later to deduce lower bounds for the general case $r \geq 3$ from the special case $r=3$. We will also use this fact when proving Theorem \ref{tight} to similarly reduce the general case $r \geq 3$ to the special case $r=3$. Finally, we state a speculation made by Day regarding conical vertices as a conjecture and discuss the evidence for this conjecture. \\

A conical vertex is a vertex adjacent to all other vertices. It is easy to check that for every graph $G$ and integer $s \geq 0$, the graph $G^s$ obtained by adding $s$ conical vertices to $G$ is $K_{r+s}$-saturated if and only if $G$ is $K_r$-saturated. We thus obtain the following two lemmas, which we will use later.

\begin{lemma}\label{c ineq}
For all integers $r \geq 3$, $t \geq r-2$ and $s \geq 0$,

$$ c(r+s,t+s) \geq c(r,t)+ts+\binom{s+1}{2} \ . $$
\end{lemma}

\begin{lemma}\label{f ineq}
For all integers $r \geq 3$, $t \geq r-2$ and $s \geq 0$,

$$ f(r+s,t+s) \geq f(r,t) \ . $$
\end{lemma}

Day speculated that perhaps for all fixed $r \geq 4$ and $t \geq r-2$ and large enough $n$, all extremal graphs for $sat(n,K_r,t)$ have a conical vertex (last paragraph in \cite{Reference 4}). This is equivalent to the statement that every graph in $C(r,t)$ has a conical vertex for $r \geq 4$. The author believes this to be true.

\begin{conj}\label{con vert}

For all integers $r \geq 4$ and $t \geq r-2$, every graph in $C(r,t)$ has a conical vertex.

\end{conj}

There is some evidence for Conjecture \ref{con vert}. Hajnal showed that every $K_r$-saturated graph $G$ with $\delta(G) < 2(r-2)$ has a conical vertex (Theorem 1 in \cite{Reference 6}). Hence Conjecture \ref{con vert} is true for $t<2(r-2)$. Moreover, Alon, Erd\"{o}s, Holzman and Krivelevich showed that Conjecture \ref{con vert} is true when $r=4=t$, the first case not covered by Hajnal's result (Theorem 8 and Corollary 3 in \cite{Reference 7}). \\

An equivalent way of stating Conjecture \ref{con vert} is that for all integers $r \geq 4$ and $t \geq r-2$, $C(r,t)$ is the collection of graphs that can be obtained by adding a conical vertex to a graph in $C(r-1,t-1)$. This would imply that we have equality in Lemma \ref{c ineq}, or equivalently that

$$ c(r,t)= c(3,t-(r-3))+(r-3)t-\binom{r-3}{2}  $$

for all integers $r \geq 3$ and $t \geq r-2$. \\

By Theorem \ref{c est}, $c(3,t) \asymp 4^t t^{-1/2}$, so Conjecture \ref{con vert} would imply 

$$ c(r,t) \asymp \frac{4^{t-r}}{\sqrt{t-r+3}} + r^2 \ , $$

or in other words, that the lower bound in Theorem \ref{c est} is tight. (Indeed, we will prove the lower bound in Theorem \ref{c est} by first using the graph $G_t$ from Section \ref{constr} to show $c(3,t) \gg 4^t t^{-1/2}$ and then using Lemma \ref{c ineq}.) So the fact that the upper bound in Theorem \ref{c est} comes close to showing the lower bound is tight is further evidence for Conjecture \ref{con vert}.

\section{Main results}\label{mainresults}

In this section we prove Theorems \ref{c est} through \ref{tight}. We first state some results that we will need in the proofs. To state the first result, we need the following definition.

\begin{definition}\label{set}

For integers $a,b \geq 0$, an $(a,b)$ \emph{set system} is a sequence $(A_i,B_i)_{i \in I}$ of pairs of finite sets, indexed by a finite, ordered set $I$, with the following properties.

\begin{enumerate}
    \item  $|A_i| \leq a$ and $|B_i| \leq b$ for all $i \in I$.
    \item $A_i \cap B_i = \emptyset  $ for all $i \in I$.
    \item $A_i \cap B_j \neq \emptyset $ for all $i < j \in I$.
\end{enumerate}

\end{definition}

\theoremstyle{definition}
\newtheorem*{repdef1}{Definition \ref{set}}
\theoremstyle{plain}

The first result is the following known version of the Two Families Theorem (see \cite{Reference 8}).

\begin{theorem}[Skew Two Families Theorem, Frankl, 1982]\label{Skew}

Let $a,b \geq 0$ be integers and $(A_i,B_i)_{i \in I}$ be an $(a,b)$ set system. Then $|I| \leq \binom{a+b}{a}$.

\end{theorem}

\newtheorem*{repth6}{Theorem \ref{Skew}}

To state the next result, we need the following definition.

\begin{definition}\label{mod set}

For integers $a,b,c \geq 0$, an $(a,b,c)$ \emph{modified set system} is a sequence $(A_i,B_i)_{i \in I}$ of pairs of finite sets, indexed by a finite, ordered set $I$, with the following properties.

\begin{enumerate}
    \item  $|A_i \cap \cup_{k<i} \ A_k| \leq a$ and $|B_i| \leq b$ for all $i \in I$.
    \item $|A_i \cap B_i| \leq c $ for all $i \in I$.
    \item $|A_i \cap B_j| > c $ for all $i < j \in I$.
\end{enumerate}

\end{definition}

\theoremstyle{definition}
\newtheorem*{repdef2}{Definition \ref{mod set}}
\theoremstyle{plain}

The result consists of Theorem \ref{TF} and two further results covering the remaining, degenerate cases $b > c \geq a$ and $b \leq c$.

\begin{theorem}\label{estimates}

Let $a,b,c \geq 0$ be integers and $(A_i,B_i)_{i \in I}$ be an $(a,b,c)$ modified set system. Then in each of the following cases we have the following bounds for $|I|$.

\begin{enumerate}
    \item  If $a,b > c$,
    
     $$ |I| \ll \binom{a+b-2c+1}{a-c+1} \ . $$
     
    \item If $b > c \geq a$,
    
    $$ |I| \leq \left\lfloor \frac{b-a}{c-a+1} \right\rfloor +1 \ . $$
    
    \item If $b \leq c$, 
    
    $$ |I| \leq 1 \ . $$ 
    
\end{enumerate}

\end{theorem}

We will prove Theorem \ref{estimates} in section \ref{TFT}. \\

A key step in the proofs of Theorems \ref{c est} and \ref{vert bound} is defining a sequence $(A_i,B_i)_{i \in I}$ of pairs of finite sets with properties similar to those in Definition \ref{mod set}, but with cardinality replaced by clique number in properties 2 and 3. The following graph theory lemma will allow us to replace these conditions by those in Definition \ref{mod set}, after which we can apply Theorem \ref{estimates}.

\begin{lemma}\label{Graph Theory Lemma}

Let $G$ be a graph. Then there is a subset $S \subseteq V(G)$ such that $|S|-w(S)=|G|-w(G)$ and $|S| \geq 2w(S)$.

\end{lemma}

\begin{proof}

Consider the subsets $S \subseteq V(G)$ satisfying $|S|-w(S)=|G|-w(G)$. Such sets do exist since we can take $S = V(G)$. Take a minimal such $S$. Then the intersection of the maximum cliques of $G[S]$ must be empty, for if there was a vertex $v$ in every maximum clique of $G[S]$, we would have $|S \setminus \{v\}| -w(S \setminus \{v\})=|G|-w(G)$, contradicting the minimality of $S$. Hajnal showed (the lemma in \cite{Reference 6}) that for any graph $H$ and non-empty collection $C$ of maximum cliques in $H$, we have $|\cup_{T \in C} T | + |\cap_{T \in C} T| \geq 2w(H)$. In particular, for any graph $H$ we have $u(H)+i(H) \geq 2w(H)$, where $u(H)$ and $i(H)$ are the size of the union and intersection of all the maximum cliques of $H$, respectively. Hence, if $i(H)=0$, we have $ |H| \geq u(H) \geq 2w(H)$. So $|S| \geq 2w(S)$.

\end{proof}

We will deduce Theorems \ref{c est} and \ref{blow-up} from the following technical lemma.

\begin{lemma}\label{tech}

Let $r \geq 3$ and $t \geq r-2$ be integers and $G$ be a $K_r$-saturated graph. Then there exists a subset $S \subseteq V(G)$ with the following properties.

\begin{enumerate}

    \item $ |S| = O_{t}(1) \ . $ 
    
    \item $$ e(S) \geq t|S| - O\left(  \frac{4^{t-r} \ \min{(r,\sqrt{t-r+3})}}{\sqrt{t-r+3}} + r^2  \right) \ .$$
    
    \item For all $v \in V(G)$, either $|\Gamma(w) \cap S| > t$ for all $ w \in \Gamma(v) \setminus S$ or  $$ |\Gamma(v) \cap S| + \frac{1}{2} |\Gamma(v) \setminus S| > t \ .  $$ 
    
\end{enumerate}

\end{lemma}

\begin{rem}

Since for every integer $t$, there are only finitely many integers $r$ satisfying $r \geq 3$ and $t \geq r-2$, property 1 is equivalent to the statement $|S|=O_{r,t}(1)$.

\end{rem}

\begin{proof}

If $|G| \leq r-2$, we can take $S=V(G)$, so suppose $|G| \geq r-1$. Consider all sequences $(v_i)_{i \in I}$ of vertices in $G$, indexed by some finite, ordered set $I$, with the following properties.

\begin{enumerate}[(a)]

    \item  For all $i \in I$, $ \ | \Gamma(v_i) \cap \cup_{k<i} \ \Gamma(v_k) | + \frac{1}{2} \ |\Gamma(v_i) \setminus \cup_{k<i} \ \Gamma(v_k) | \leq t  \ . $ 
    
    \item For all $i \in I$, there exist vertices $w \in \Gamma(v_i) \setminus \cup_{k<i} \ \Gamma(v_k)$ such that $ |\Gamma(w) \cap \cup_{k<i} \ \Gamma(v_k) | \leq t \ $.

\end{enumerate}

Such sequences do exist since we can take the empty sequence. Note that property (b) forces the $v_i$ to be distinct. Hence we can take a maximal such $(v_i)_{i \in I}$, in the sense that one cannot adjoin a vertex at the end of the sequence and obtain a sequence with the same properties. Let $S = \cup_{i \in I} \ \Gamma(v_i)$. Then property 3 holds by the maximality of the sequence. \\

For each $i \in I$, pick a vertex $w_i \in \Gamma(v_i) \setminus \cup_{k<i} \ \Gamma(v_k)$ that minimises $|\Gamma(w_i) \cap \cup_{k<i} \ \Gamma(v_k) |$. Let $A_i = \Gamma(v_i)$ and $B_i = \Gamma(w_i) \cap \cup_{k<i} \ \Gamma(v_k)$. Then property (a) can be rewritten as

\begin{equation}\label{e1}
| A_i \cap \cup_{k<i} \ A_k | + \frac{1}{2} \ |A_i \setminus \cup_{k<i} \ A_k | \leq t 
\end{equation}

for all $i \in I$ and property (b) is equivalent to 

\begin{equation}\label{e2}
|B_i| \leq t  
\end{equation}

for all $i \in I$. \\

We now show that $(A_i,B_i)_{i \in I}$ has properties 2 and 3 in Definition \ref{mod set}, but with cardinality replaced by clique number. More precisely, we have

\begin{equation}\label{e3}
w(A_i \cap B_i) \leq r-3 
\end{equation}

for all $i \in I$ and 

\begin{equation}\label{e4}
w(A_i \cap B_j) > r-3
\end{equation}

for all $i<j \in I$. To prove (\ref{e3}), note that since $v_i$ and $w_i$ are adjacent and $G$ is $K_r$-free, $G[\Gamma(v_i) \cap \Gamma(w_i)]$ is $K_{r-2}$-free. But $A_i \cap B_i \subseteq \Gamma(v_i) \cap \Gamma(w_i)$, so (\ref{e3}) follows. To prove (\ref{e4}), note that since $v_i$ and $w_j$ are non-adjacent and $G$ is $K_r$-saturated, $G[\Gamma(v_i) \cap \Gamma(w_j)]$ contains a $K_{r-2}$ (since $|G| \geq r-1$, this still holds when $v_i=w_j$). But $A_i \cap B_j = \Gamma(v_i) \cap \Gamma(w_j) $, so (\ref{e4}) follows. \\

Let us now show property 1 holds. Since this only requires a qualitative bound, we can afford to be less careful when bounding quantities. We have $S = \cup_{i \in I} \ A_i$. By (\ref{e1}), $|A_i| \leq 2t$ for all $i \in I$, so it suffices to show $|I|=O_{t}(1)$. Combining (\ref{e3}) and (\ref{e4}) gives $A_i \cap B_j \not \subseteq A_i \cap B_i$ for all $i < j \in I$, which is equivalent to $C_i \cap B_j \neq \emptyset$ for all $i < j \in I$, where $C_i = A_i \setminus B_i$ for each $i \in I$. We also have $|C_i| \leq |A_i| \leq 2t$, $|B_i| \leq t$ by (\ref{e2}) and $C_i \cap B_i = \emptyset$ for all $i \in I$. Hence $(C_i,B_i)_{i \in I}$ is a $(2t,t)$ set system, so $|I| \leq \binom{3t}{t}$ by Theorem \ref{Skew}. \\

Let us now show property 2 holds. Since this requires a quantitative bound, we will need to be more careful when bounding quantities. We have

$$ |S| = |\cup_{i \in I}   A_i| = \sum_{i \in I} |A_i \setminus \cup_{k<i} \ A_k| $$ 

and

$$ e(S) = e(\cup_{i \in I} \  A_i) \geq \sum_{i \in I} e\left( \cup_{k<i} \ A_k , \ A_i \setminus \cup_{k<i} \ A_k \right) \geq \sum_{i \in I} |A_i \setminus \cup_{k<i} \ A_k| \ |B_i|  $$
 
by the minimality of $w_i$. Hence

\begin{equation}\label{e5}
t|S|-e(S) \leq \sum_{i \in I} |A_i \setminus \cup_{k<i} \ A_k| \ (t-|B_i|) \leq 2 \sum_{i \in I} (t-|A_i \cap \cup_{k<i} \ A_k|) \ (t-|B_i|) 
\end{equation} 

by (\ref{e1}) and (\ref{e2}). \\

Our aim now is to show

\begin{equation}\label{e6}
\sum_{i \in I} (t-|A_i \cap \cup_{k<i} \ A_k|) \ (t-|B_i|) \ll  \frac{4^{t-r} \ \min{(r,\sqrt{t-r+3})}}{\sqrt{t-r+3}} + r^2 \ ,
\end{equation} 

which together with (\ref{e5}) proves property 2 holds. Note that (\ref{e1}) implies

\begin{equation}\label{e7}
| A_i \cap \cup_{k<i} \ A_k | \leq t
\end{equation}

for all $i \in I$. We would like to use (\ref{e7}), (\ref{e2}), (\ref{e3}) and (\ref{e4}) to apply Theorem \ref{estimates} to $(A_i,B_i)_{i \in I}$ to prove (\ref{e6}), but there are two obstacles. The first is that in (\ref{e3}) and (\ref{e4}), we have bounds for the clique numbers instead of the cardinalities. The second is that instead of an upper bound for $|I|=\sum_{i \in I} \ 1$, we need an upper bound for the sum in (\ref{e6}), which involves the weights $(t-|A_i \cap \cup_{k<i} \ A_k|) \ (t-|B_i|)$. \\

We first use Lemma \ref{Graph Theory Lemma} to overcome the first problem, as follows. For each $i \in I$, apply Lemma \ref{Graph Theory Lemma} to $G[A_i \cap B_i]$ to obtain a set $D_i \subseteq A_i \cap B_i$ such that

\begin{equation}\label{e8}
|D_i|-w(D_i) = |A_i \cap B_i| - w(A_i \cap B_i)
\end{equation}

and

\begin{equation}\label{e9}
|D_i| \geq 2w(D_i) \ .
\end{equation}

Let $E_i = B_i \setminus D_i$. We then show that we have

\begin{equation}\label{e10}
|E_i| \leq |B_i| - 2w(D_i)
\end{equation}

and

\begin{equation}\label{e11}
|A_i \cap E_i| \leq r - 3 - w(D_i)
\end{equation}

for all $i \in I$ and

\begin{equation}\label{e12}
|A_i \cap E_j| > r-3-w(D_j) 
\end{equation}

for all $i<j \in I$. For (\ref{e10}), note that

$$ |E_i| = |B_i|-|D_i| \leq |B_i|-2w(D_i) $$

by (\ref{e9}). For (\ref{e11}), note that 

$$|A_i \cap E_i| = |A_i \cap B_i| - |D_i| = w(A_i \cap B_i) - w(D_i) \leq r - 3 - w(D_i)$$
 
by (\ref{e8}) and (\ref{e3}). Finally, for (\ref{e12}), note that

$$|A_i \cap E_j| \geq w(A_i \cap E_j) \geq w(A_i \cap B_j) - w(D_j) > r-3-w(D_j)  $$ 

by (\ref{e4}). \\

We now address the problem of the weights. For all integers $1 \leq a,b \leq t $, let $I_{a,b} = \{ i \in I : \ t-|A_i \cap \cup_{k<i} \ A_k| \geq a, \ t-|B_i| \geq b \}$. By double counting the number of triples $(i,a,b)$, where $i \in I$ and $a$ and $b$  are integers, with $1 \leq a \leq t-|A_i \cap \cup_{k<i} \ A_k|$ and $1 \leq b \leq t-|B_i|$ (and using (\ref{e2}) and (\ref{e7})), we obtain
 
\begin{equation}\label{e13}
\sum_{i \in I} (t-|A_i \cap \cup_{k<i} \ A_k|) \ (t-|B_i|) = \sum_{\substack{1 \leq a \leq t \\ 1 \leq b \leq t}} |I_{a,b}| \ .
\end{equation}

%$$  \sum_{i>f} (t-|A_i \cap \ \cup_{j<i} \ A_j|) \ (t-|B_i|) =  \ .  $$

Note that for $i \in I_{a,b}$, we have

\begin{equation}\label{e14}
| A_i \cap \cup_{k<i} \ A_k | \leq t-a \ 
\end{equation}

and

\begin{equation}\label{e15}
|E_i| \leq t-b-2w(D_i) 
\end{equation}

by (\ref{e10}). \\

We would now like to use (\ref{e14}), (\ref{e15}), (\ref{e11}) and (\ref{e12}) to apply Theorem \ref{estimates} to $(A_i,E_i)_{i \in I_{a,b}}$ to bound $|I_{a,b}|$, but cannot since the bounds depend on the $w(D_i)$. We would like to overcome this by partitioning $I_{a,b}$ into parts $I_{a,b,c}$ depending on the value of $w(D_i)$. Then depending on the values of $a$, $b$ and $c$, we would bound $|I_{a,b,c}|$ using case 1, 2 or 3 of Theorem \ref{estimates}. Note that by doing this we discard the information that (\ref{e12}) holds for $i$ and $j$ in different parts. The contributions from case 3 turn out to be the ones that ultimately give rise to the $r^2$ term in (\ref{e6}), but since we would have to sum over all possible values of $c$, we would obtain an inferior bound of $r^3$ instead. So instead we deal with case 3 before partitioning $I_{a,b}$ and discarding information, as follows. \\

If $I$ is empty, the sum in (\ref{e6}) is 0, so suppose otherwise. Let $f$ be the first element of $I$. Note that $f \in I_{a,b}$ for all $a$ and $b$. Then for all $f<i \in I_{a,b}$, we have 

$$ r-3-w(D_i) < |A_f \cap E_i| \leq |E_i| \leq t-b-2w(D_i) $$

by (\ref{e12}) and (\ref{e15}), which implies 

\begin{equation}\label{e16}
w(D_i) \leq t-r+2-b \ .
\end{equation}

For all integers $1 \leq a,b \leq t$ and $0 \leq c \leq \min{(r-3, \ t-r+2-b)}$, let $I_{a,b,c} = \{ i \in I_{a,b} \setminus \{f\} : \ w(D_i)=c \}$, with the order induced by $I$. Note that the $I_{a,b,c}$ partition $I_{a,b} \setminus \{f\}$ for all $a$ and $b$ by (\ref{e3}) and (\ref{e16}). By (\ref{e14}), (\ref{e15}), (\ref{e11}) and (\ref{e12}), $(A_i,E_i)_{i \in I_{a,b,c}}$ is a $(t-a, \ t-b-2c, \ r-3-c)$ modified set system for all $a$, $b$ and $c$, so when $t-a > r-3-c$, we have 

\begin{equation*}
|I_{a,b,c}| \ll \binom{2t-2r+7-a-b}{t-r+4-a+c}
\end{equation*}

by part 1 of Theorem \ref{estimates} and when $t-a \leq r-3-c$, we have

\begin{equation*}
|I_{a,b,c}| \leq \left\lfloor \frac{a-b-2c}{a-c-t+r-2} \right\rfloor + 1
\end{equation*}

by part 2 of Theorem \ref{estimates}. (Note that the $\cup_{k<i} \ A_k$ may change when restricting to $I_{a,b,c}$, but can only become smaller.) \\

Hence

$$ \sum_{i \in I} (t-|A_i \cap \cup_{k<i} \ A_k|) \ (t-|B_i|) = \sum_{\substack{1 \leq a \leq t \\ 1 \leq b \leq t}}  |I_{a,b}| = \sum_{\substack{1 \leq a \leq t \\ 1 \leq b \leq t}} \left( \sum_{\substack{0 \leq c \leq r-3 \\ c \leq t-r+2-b}} |I_{a,b,c}| \ + \ 1 \right)  $$

$$ \ll \sum_{\substack{1 \leq a \leq t \\ 1 \leq b \leq t \\ 0 \leq c \leq r-3 \\ c \leq t-r+2-b \\ t-a > r-3-c}} \binom{2t-2r+7-a-b}{t-r+4-a+c} \ + \sum_{\substack{1 \leq a \leq t \\ 1 \leq b \leq t \\ 0 \leq c \leq r-3 \\ c \leq t-r+2-b \\ t-a \leq r-3-c}} \left( \left\lfloor \frac{a-b-2c}{a-c-t+r-2} \right\rfloor + 1 \right) \ + \ t^2 \  $$

by (\ref{e13}), where the first and second sum consist of the contributions from case 1 and 2, respectively. The contribution of the second sum turns out to be negligible, so we can afford to be less careful when bounding it. Each term in the sum is at most $t-r+3$ and the sum is over at most $t$ values of $a$, $t-r+3$ values of $b$ and $t-r+3$ values of $c$, so the sum is at most $t(t-r+3)^3$. \\

For the first sum, we have

$$ \sum_{\substack{1 \leq a \leq t \\ 1 \leq b \leq t \\ 0 \leq c \leq r-3 \\ c \leq t-r+2-b \\ t-a > r-3-c}} \binom{2t-2r+7-a-b}{t-r+4-a+c} \ \leq \sum_{\substack{1 \leq a \leq t \\ 0 \leq c \leq r-3 \\ c \leq t-r+1 \\ t-a > r-3-c}} \binom{2t-2r+7-a}{t-r+5-a+c} $$

$$ \leq \sum_{\substack{0 \leq c \leq r-3 \\ c \leq t-r+1}} \binom{2t-2r+7}{t-r+3-c} \ \ll \ \frac{4^{t-r} \ \min{(r,\sqrt{t-r+3})}}{\sqrt{t-r+3}}  \ . $$

Putting everything together, we obtain

$$ \sum_{i \in I} (t-|A_i \cap \cup_{k<i} \ A_k|) \ (t-|B_i|) \ \ll \  \frac{4^{t-r} \ \min{(r,\sqrt{t-r+3})}}{\sqrt{t-r+3}} \ + \ t(t-r+3)^3 \ + \ t^2  $$

$$ \ \asymp \ \frac{4^{t-r} \ \min{(r,\sqrt{t-r+3})}}{\sqrt{t-r+3}} \ + \ r^2  \ , $$

which proves (\ref{e6}).

\end{proof}

We now deduce Theorem \ref{c est} from Lemma \ref{tech}. In fact, we prove a slightly stronger result.

\begin{repth1}

For all integers $r \geq 3$ and $t \geq r-2$,

$$ \frac{4^{t-r}}{\sqrt{t-r+3}} + r^2  \ll c(r,t) \leq c'(r,t) \ll  \frac{4^{t-r} \ \min{(r,\sqrt{t-r+3})}}{\sqrt{t-r+3}} + r^2 \ . $$

\end{repth1}

\begin{proof}

We first prove the lower bound. By Lemma \ref{c ineq}, it is sufficient to prove the special case $r=3$, which states that $c(3,t) \gg 4^t t^{-1/2}$ for all integers $t \geq 1$. By properties 1 and 4 in section \ref{constr}, for all integers $t \geq 3$, $G_t$ is $K_3$-saturated and $\delta(G_t)=t$, respectively. Hence, by properties 2 and 3 in section \ref{constr},

$$ c(3,t) \geq t|G_t|-e(G_t) = \frac{1}{2} \binom{2(t-1)}{t-1} + 2t(t-1) \asymp 4^t t^{-1/2}  $$

for all integers $t \geq 3$. It is easy to check that $c(3,1), c(3,2) > 0$ (in fact, as we shall see in section \ref{known}, $c(3,1)=1$ and $c(3,2)=5$), so $c(3,t) \gg 4^t t^{-1/2}$ for all integers $t \geq 1$. \\

We now prove the upper bound. Let $G$ be a $K_r$-saturated graph with $ \delta(G) \geq t$. We need to show that

$$ e(G) \geq t|G| - O\left(  \frac{4^{t-r} \ \min{(r,\sqrt{t-r+3})}}{\sqrt{t-r+3}} + r^2 \ \right) \  . $$

By Lemma \ref{tech}, there exists a subset $S \subseteq V(G)$ with the properties stated in the lemma. Let $T=S \cup \{v \in V(G) : \ |\Gamma(v) \cap S| > t \}$. Then

$$ e(T) \geq e(S) + e(S,T \setminus S) $$ 

$$ \geq t|S| -  O\left(  \frac{4^{t-r} \ \min{(r,\sqrt{t-r+3})}}{\sqrt{t-r+3}} + r^2  \right) + t |T \setminus S| = t |T| -  O\left(  \frac{4^{t-r} \ \min{(r,\sqrt{t-r+3})}}{\sqrt{t-r+3}} + r^2  \right)   $$

by property 2 in the lemma and the definition of $T$. \\

We now claim that for all $v \in V(G)$ we have

$$ |\Gamma(v) \cap T| + \frac{1}{2} |\Gamma(v) \setminus T| \geq t \ . $$

Indeed, by property 3 in the lemma, either $|\Gamma(w) \cap S| > t$ for all $w \in \Gamma(v) \setminus S$ or 

$$ |\Gamma(v) \cap S| + \frac{1}{2} |\Gamma(v) \setminus S| > t \ . $$

In the first case we have $\Gamma(v) \subseteq T$, so 

$$ |\Gamma(v) \cap T| + \frac{1}{2} |\Gamma(v) \setminus T| = d(v) \geq t \ , $$

since  $\delta(G) \geq t$. In the second case we have

$$  |\Gamma(v) \cap T| + \frac{1}{2} |\Gamma(v) \setminus T| \geq  |\Gamma(v) \cap S| + \frac{1}{2} |\Gamma(v) \setminus S| > t \ , $$

since $S \subseteq T$. \\

Hence

$$ e(G) = e(T) + \sum_{v \in V(G) \setminus T} \left( |\Gamma(v) \cap T| + \frac{1}{2} |\Gamma(v) \setminus T| \right) $$

$$ \geq t |T| -   O\left(  \frac{4^{t-r} \ \min{(r,\sqrt{t-r+3})}}{\sqrt{t-r+3}} + r^2  \right) + t |V(G) \setminus T| = t|G| - O\left(  \frac{4^{t-r} \ \min{(r,\sqrt{t-r+3})}}{\sqrt{t-r+3}} + r^2  \right) \ .  $$

\end{proof}

We now deduce Theorem \ref{blow-up}, in its equivalent form, from Lemma \ref{tech}.

\begin{repth2}

For all integers $r \geq 3$, $t \geq r-2$ and $k$, $C(r,t,k)$ is finite. 

\end{repth2}

\begin{proof}

Let $G \in C(r,t,k)$. We need to show that $|G|=O_{r,t,k}(1)$. By Lemma \ref{tech}, there is a subset $S \subseteq V(G)$ with the properties stated in the lemma. Let $T=S \cup \{v \in V(G) : \ |\Gamma(v) \cap S| > t \}$ as before. Then

$$ e(T) \geq e(S) + e(S,T \setminus S) $$ 

$$ \geq t|S| -  O_{r,t}(1) + (t+1) |T \setminus S| = t |T| -  O_{r,t}(1) +  |T \setminus S|  $$

by property 2 in the lemma and the definition of $T$. \\

We now claim that for all $v \in V(G)$ we have

$$ |\Gamma(v) \cap T| + \frac{1}{2} |\Gamma(v) \setminus T| \geq t \ , $$

with equality if and only if $\Gamma(v) \subseteq T$ and $d(v)=t$. Indeed, by property 3 in the lemma, either $|\Gamma(w) \cap S| > t$ for all $w \in \Gamma(v) \setminus S$ or 

$$ |\Gamma(v) \cap S| + \frac{1}{2} |\Gamma(v) \setminus S| > t \ . $$

In the first case we have $\Gamma(v) \subseteq T$ and

$$ |\Gamma(v) \cap T| + \frac{1}{2} |\Gamma(v) \setminus T| = d(v) \geq t \ , $$

as before. In the second case we have

$$  |\Gamma(v) \cap T| + \frac{1}{2} |\Gamma(v) \setminus T| \geq  |\Gamma(v) \cap S| + \frac{1}{2} |\Gamma(v) \setminus S| > t \ , $$

as before.  \\

Let $ R = \{ v \in V(G) : \ \Gamma(v) \subseteq T, d(v)=t \} $. We then have

$$ t|G|+k \ = \ e(G) \ = \ e(T) + \sum_{v \in V(G) \setminus T} \left( |\Gamma(v) \cap T| + \frac{1}{2} |\Gamma(v) \setminus T| \right)  $$ 

$$ \geq \ t |T| - O_{r,t}(1) + |T \setminus S|  + t |V(G) \setminus T| + \frac{1}{2} |V(G) \setminus (T \cup R)| \ = \ t|G| - O_{r,t}(1) + |T \setminus S| + \frac{1}{2} |V(G) \setminus ( T \cup R)| \ .  $$

Hence $|T \setminus S|, \ |V(G) \setminus (T \cup R)| =O_{r,t,k}(1)$. We also have $|S|=O_{r,t}(1)$ by property 1 in the lemma. Combining these, we obtain $|V(G) \setminus (R \setminus T)|=O_{r,t,k}(1)$.  \\

Note that $R \setminus T$ is an independent set of vertices of degree $t$. We now claim that for every subset $N \subseteq V(G) \setminus (R \setminus T)$ there are at most $t$ vertices $v \in R \setminus T$ with $\Gamma(v) = N$. We then have $|R \setminus T| = O_{r,t,k}(1)$ and hence $|G|=O_{r,t,k}(1)$, completing the proof. Suppose for the sake of contradiction that there is a subset $N \subseteq V(G) \setminus (R \setminus T)$ such that $|U| > t$, where $U=\{v \in R \setminus T : \Gamma(v)=N  \}$. \\

Pick a subset $W \subset U$ with $|W|=t$ and let $H=G[V(G) \setminus (U \setminus W)]$. Note that $H$ is smaller than $G$ and that $G$ can be obtained from $H$ by blowing up vertices in $W$, which are non-adjacent and of degree $t$. Hence $H$ is $K_r$-saturated. The only vertices with smaller degree in $H$ than in $G$ are those in $N$. But these vertices are adjacent in $H$ to all vertices in $W$, so $\delta(H) = t$. Since $U \setminus W$ is an independent set of vertices of degree $t$, $e(H)=t|H|+k$. This contradicts the minimality of $G$.

\end{proof}

We now prove Theorem \ref{vert bound}. We first state a lemma that we will need in the proof. To state the lemma, we need the following definition.

\begin{definition}\label{deg set system}

For integers $b > c \geq a \geq 0$, an $(a,b,c)$ \emph{degenerate set system} is a pair $((A_i)_{i \in I}, B)$, where $(A_i)_{i \in I}$ is a sequence of finite sets, indexed by a finite, ordered set $I$, and $B$ is a finite set, with the following properties.

\begin{enumerate}
    \item $|B| \leq b$.
    \item $|A_i \cap B| > c $ for all $i \in I$.
    \item $|A_i \cap \cup_{k<i} \ A_k| \leq a $ for all $i \in I$.
\end{enumerate}

\end{definition}

The lemma is the following slight strengthening of case 2 of Theorem \ref{estimates}.

\begin{lemma}\label{Degenerate Cases}

Let $b > c \geq a \geq 0$ be integers and $((A_i)_{i \in I},B)$ be an $(a,b,c)$ degenerate set system. Then $|I| \leq \left\lfloor \frac{b-a}{c-a+1} \right\rfloor $.

\end{lemma}

We will prove Lemma \ref{Degenerate Cases} in section \ref{TFT}. We are now ready to prove Theorem \ref{vert bound}. The proof will be similar to that of Lemma \ref{tech}.

\begin{repth3}

For all integers $r \geq 3$ and $t \geq r-2$,

$$ \frac{4^{t-r}}{\sqrt{t-r+3}} + r \ll f(r,t) \ll \frac{4^{t-r} \min{(r,\sqrt{t-r+3})}}{\sqrt{t-r+3}} + r \ . $$

\end{repth3}

\begin{proof}

We first prove the lower bound. We need to show that $f(r,t) \gg r$ and $f(r,t) \gg 4^{t-r} / \sqrt{t-r+3}$. For the former, consider the graph $K_{r-1}$. For the latter, by Lemma \ref{f ineq}, it is sufficient to prove the special case $r=3$, which states that $f(3,t) \gg 4^t t^{-1/2}$ for all integers $t \geq 1$. By properties 1, 4 and 5 in section \ref{constr}, $f(3,t) \geq \binom{2(t-1)}{t-1} \gg 4^t t^{-1/2}$ for all integers $t \geq 3$. Since $f(r,t) \gg r$, $f(r,t)>0$, so in particular $f(3,1),f(3,2) >0$. Hence $f(3,t) \gg 4^t t^{-1/2}$ for all integers $t \geq 1$.   \\

We now prove the upper bound. Let $G$ be a $K_r$-saturated graph. We need to bound the number of vertices in $G$ of degree at most $t$ which are adjacent to another vertex of degree at most $t$. If $|G| \leq r-2$, the upper bound holds, so suppose $|G| \geq r-1$. Consider all sequences $(v_i,w_i)_{i \in I}$ of pairs of vertices in $G$, indexed by some finite, ordered set $I$, with the following properties.

\begin{enumerate}

    \item  For all $i \in I$, $d(v_i), d(w_i) \leq t \ . $  
    
    \item For all $i \in I$, $v_i$ and $w_i$ are adjacent. 
    
    \item For all $i < j \in I$, $v_i$ and $w_j$ are not adjacent.
    
\end{enumerate}

Such sequences do exist since we can take the empty sequence. Note that properties 2 and 3 force the $v_i$ to be distinct (and the $w_i$ to be distinct). Hence we can take a maximal such $(v_i,w_i)_{i \in I}$, in the sense that one cannot adjoin a pair at the end of the sequence and obtain a sequence with the same properties. Then by the maximality of the sequence, every vertex in $G$ of degree at most $t$ adjacent to another vertex of degree at most $t$ is in $\cup_{i \in I} \ \Gamma(v_i)$, so it suffices to bound  $|\cup_{i \in I} \Gamma(v_i)|$. \\

For each $i \in I$, let $A_i = \Gamma(v_i)$ and $B_i = \Gamma(w_i)$. Then property 1 can be rewritten as 

\begin{equation}\label{e20}
|A_i| \leq t
\end{equation}

and

\begin{equation}\label{e23}
|B_i| \leq t
\end{equation}

for all $i \in I$. As in Lemma \ref{tech}, we now show that $(A_i,B_i)_{i \in I}$ has properties 2 and 3 in Definition \ref{mod set}, but with cardinality replaced  by clique number. More precisely, we have

\begin{equation}\label{e18}
w(A_i \cap B_i) \leq r-3
\end{equation}

for all $i \in I$ and

\begin{equation}\label{e19}
w(A_i \cap B_j) > r-3
\end{equation}

for all $i<j \in I$. To prove (\ref{e18}), note that $v_i$ and $w_i$ are adjacent by property 2 and $G$ is $K_r$-free. For (\ref{e19}), note that $v_i$ and $w_j$ are non-adjacent by property 3 and $G$ is $K_r$-saturated (since $|G| \geq r-1$, this still holds when $v_i=w_j$). \\

We have

\begin{equation}\label{e21}
|\cup_{i \in I} \Gamma(v_i)| = |\cup_{i \in I} A_i| = \sum_{i \in I} |A_i \setminus \cup_{k < i} \ A_k| \leq \sum_{i \in I} \left(t-|A_i \cap \cup_{k<i} \ A_k| \right)
\end{equation}

by (\ref{e20}). Our aim now is to show 

\begin{equation}\label{e22}
\sum_{i \in I} \left(t-|A_i \cap \cup_{k<i} \ A_k| \right) \ll \frac{4^{t-r} \min{(r,\sqrt{t-r+3})}}{\sqrt{t-r+3}} + r \ , 
\end{equation}

which together with (\ref{e21}) proves the upper bound. Note that (\ref{e20}) implies 

\begin{equation}\label{e24}
|A_i \cap \cup_{k<i} \ A_k| \leq t
\end{equation}

for all $i \in I$. As before, we would like to use (\ref{e24}), (\ref{e23}), (\ref{e18}) and (\ref{e19}) to apply Theorem \ref{estimates} to $(A_i,B_i)_{i \in I}$ to prove (\ref{e22}), but there are two obstacles. The first is that in (\ref{e18}) and (\ref{e19}), we have bounds for the clique numbers instead of the cardinalities. The second is that instead of an upper bound for $|I|=\sum_{i \in I} \ 1$, we need an upper bound for the sum in (\ref{e22}), which involves the weights $(t-|A_i \cap \cup_{k<i} \ A_k|)$. \\
 
As in Lemma \ref{tech}, we first use Lemma \ref{Graph Theory Lemma} to overcome the first problem. For each $i \in I$, apply Lemma \ref{Graph Theory Lemma} to $G[A_i \cap B_i]$ to obtain a set $C_i \subseteq A_i \cap B_i$ such that

\begin{equation}\label{e25}
|C_i| - w(C_i) = |A_i \cap B_i| - w(A_i \cap B_i)
\end{equation}

and

\begin{equation}\label{e26}
|C_i| \geq 2w(C_i) \ .
\end{equation}

For each $i \in I$, let $D_i = B_i \setminus C_i$. We then show that we have 

\begin{equation}\label{e27}
|D_i| \leq t - 2w(C_i)
\end{equation}

and

\begin{equation}\label{e28}
|A_i \cap D_i| \leq r-3-w(C_i)
\end{equation}

for all $i \in I$ and

\begin{equation}\label{e29}
|A_i \cap D_j| > r-3-w(C_j)
\end{equation}

for all $i<j \in I$. For (\ref{e27}), note that

$$ |D_i|=|B_i|-|C_i| \leq  t - 2w(C_i) $$

by (\ref{e23}) and (\ref{e26}). For (\ref{e28}), note that

$$ |A_i \cap D_i| = |A_i \cap B_i| - |C_i| = w(A_i \cap B_i) - w(C_i) \leq r - 3 - w(C_i) $$

by (\ref{e25}) and (\ref{e18}). Finally, for (\ref{e29}), note that 

$$|A_i \cap D_j| \geq w(A_i \cap D_j) \geq w(A_i \cap B_j) - w(C_j) > r-3-w(C_j)$$

by (\ref{e19}). \\

We now address the problem of the weights as before. For all integers $1 \leq a \leq t$, let $I_a=\{i \in I : \ t-|A_i \cap \cup_{k<i} \ A_k| \geq a \}$. By double counting the number of pairs $(i,a)$, where $i \in I$ and $a$ is an integer, with $1 \leq a \leq t-|A_i \cap \cup_{k<i} \ A_k|$ (and using (\ref{e24})), we obtain

\begin{equation}\label{e69}
\sum_{i \in I} \left(t-|A_i \cap \cup_{k<i} \ A_k| \right) = \sum_{1 \leq a \leq t} |I_a| \ .
\end{equation} 

Note that for $i \in I_a$, we have

\begin{equation}\label{e30}
|A_i \cap \cup_{k<i} \ A_k| \leq t-a \ .
\end{equation}

As in Lemma \ref{tech}, we would now like to use (\ref{e30}), (\ref{e27}), (\ref{e28}) and (\ref{e29}) to apply Theorem \ref{estimates} to $(A_i,D_i)_{i \in I_a}$ to bound $|I_a|$, but cannot since the bounds depend on the $w(C_i)$. We would like to overcome this by partitioning $I_a$ into parts $I_{a,b}$ depending on the value of $w(C_i)$. Then depending on the values of $a$ and $b$, we would bound $|I_{a,b}|$ using case 1, 2 or 3 of Theorem \ref{estimates}. \\

As before, the contributions from case 3 turn out to be the ones that ultimately give rise to the $r$ term in (\ref{e22}), but since we would have to sum over all possible values of $b$, we would obtain an inferior bound of $r^2$ instead, so we instead deal with case 3 before partitioning $I_{a}$. Unlike in Lemma \ref{tech}, however, the contributions from case 2 would turn out to not be negligible and indeed would ultimately give an inferior bound of $r \log r$ instead of $r$, so this time we need to also deal with case 2 before partitioning $I_{a}$.  \\

We first deal with case 3 as in Lemma \ref{tech}. If $I$ is empty, the sum in (\ref{e22}) is 0, so suppose otherwise. Let $f$ be the first element of $I$. Note that $f \in I_a$ for all $a$. Then for all $f<i \in I$, we have

$$r-3-w(C_i) < |A_f \cap D_i| \leq |D_i| \leq t-2w(C_i) $$

by (\ref{e29}) and (\ref{e27}), which implies

\begin{equation}\label{e31}
w(C_i) \leq t-r+2 \ .
\end{equation}

\leavevmode\newline We now deal with case 2 using Lemma \ref{Degenerate Cases}, depending on the value of $a$. For $a \geq 2t-2r+5$, there are no $i \in I_a$ in case 1, so the only remaining case is case 2. For such $a$ we show that $|I_a| \leq 2$, as follows. Let $l$ be the last element of $I_a$ in the order induced by $I$. If $l=f$, $I_a=\{f\}$, so suppose $l>f$. Then $w(C_l) \leq t-r+2$ by (\ref{e31}). Order $I_a \setminus \{l\}$ using the order induced by $I$. Then by (\ref{e27}), (\ref{e29}) and (\ref{e30}), $((A_i)_{i \in I_a \setminus \{l\}},\ D_l)$ is a $(t-a,\ t-2w(C_l), \ r-3-w(C_l))$ degenerate set system, so

$$ |I_a| = |I_a \setminus \{l\}| + 1 \leq \left \lfloor \frac{a-2w(C_l)}{a-w(C_l)-t+r-2} \right \rfloor + 1 = 2  $$

by Lemma \ref{Degenerate Cases}. (Note that the $\cup_{k<i} \ A_k$ may change when restricting to $I_a \setminus \{l\}$, but can only become smaller.)

\leavevmode\newline For $a \leq 2t-2r+4$, let $J_a = \{ i \in I_a : \ w(C_i) \leq a-t+r-3 \}$ be the set of $i \in I_a$ in case 2. We claim that $|J_a| \leq 2t-2r+7-a$. Indeed, if $|J_a| \leq 1$, the bound holds, so suppose $|J_a| \geq 2$. Let $l$ be the last element of $J_a$ in the order induced by $I$. Since $|J_a| \geq 2$, $l>f$, so $w(C_l) \leq t-r+2$ by (\ref{e31}). Order $J_a \setminus \{l\}$ using the order induced by $I$. Then, as before, by (\ref{e27}), (\ref{e29}) and (\ref{e30}), $((A_i)_{i \in J_a \setminus \{l\}}, \ D_l)$ is a $(t-a, \ t-2w(C_l), \ r-3-w(C_l))$ degenerate set system, so

$$ |J_a| = |J_a \setminus \{l\}| + 1 \leq \left \lfloor \frac{a-2w(C_l)}{a-w(C_l)-t+r-2} \right \rfloor + 1 \leq 2t-2r+7-a $$

by Lemma \ref{Degenerate Cases}. (Again, the $\cup_{k<i} \ A_k$ may change when restricting to $J_a \setminus \{l\}$, but can only become smaller.)

\leavevmode\newline We now bound $|I_a \setminus J_a|$. For each integer $\max{(0, \ a-t+r-2)} \leq b \leq \min{(r-3, \ t-r+2)}$, let $I_{a,b} = \{ i \in (I_a \setminus J_a) \setminus \{f\} : \ w(C_i) = b \}$, with the order induced by $I$. Note that the $I_{a,b}$ partition $(I_a \setminus J_a) \setminus \{f\}$ by (\ref{e18}) and (\ref{e31}). Then by (\ref{e30}), (\ref{e27}), (\ref{e28}) and (\ref{e29}), $(A_i, D_i)_{i \in I_{a,b}}$ is a $(t-a, \ t-2b, \ r-3-b)$ modified set system, so 

$$ |I_{a,b}| \ll \binom{2t-2r+7-a}{t-r+4-a+b} $$

by part 1 of Theorem \ref{estimates}. (Once again, the $\cup_{k<i} \ A_k$ may change when restricting to $I_{a,b}$, but can only become smaller.) \\

Hence
 
$$ |I_a \setminus J_a| \ \leq \ |(I_a \setminus J_a) \setminus \{f\}| \ + \ 1 \ \ = \sum_{\substack{0 \leq b \leq r-3 \\ a-t+r-2 \leq b \\ b \leq t-r+2}} |I_{a,b}| \ + \ 1 \ \ll \sum_{\substack{0 \leq b \leq r-3 \\ a-t+r-2 \leq b \\ b \leq t-r+2}} \binom{2t-2r+7-a}{t-r+4-a+b} \ + \ 1 \ , $$

so 
 
$$ |I_a| \ = \ |I_a \setminus J_a| + |J_a| \ \ll \sum_{\substack{0 \leq b \leq r-3 \\ a-t+r-2 \leq b \\ b \leq t-r+2}} \binom{2t-2r+7-a}{t-r+4-a+b} \ + \ (2t-2r+8-a) $$   

for $a \leq 2t-2r+4$. \\

Combining our upper bounds for the $|I_a|$, we obtain

$$ \sum_{i \in I} \left(t-|A_i \cap \cup_{k<i} \ A_k| \right) \ = \ \sum_{1 \leq a \leq t} |I_a| $$

$$ \ll \sum_{\substack{1 \leq a \leq t \\ a \leq 2t-2r+4}} \left[ \sum_{\substack{0 \leq b \leq r-3 \\ a-t+r-2 \leq b \\ b \leq t-r+2}} \binom{2t-2r+7-a}{t-r+4-a+b} \ \ + \ (2t-2r+8-a) \right] \ + \ \sum_{\substack{a \geq 2t-2r+5 \\ a \leq t}} 2  $$ 

$$ \ll \sum_{\substack{1 \leq a \leq t \\ a \leq 2t-2r+4 \\ 0 \leq b \leq r-3 \\ a-t+r-2 \leq b \\ b \leq t-r+2 }} \binom{2t-2r+7-a}{t-r+4-a+b} \ + \ (t-r+3)^2 \ + \ t  $$

by (\ref{e69}). \\

We have

$$ \sum_{\substack{1 \leq a \leq t \\ a \leq 2t-2r+4 \\ 0 \leq b \leq r-3 \\ a-t+r-2 \leq b \\ b \leq t-r+2 }} \binom{2t-2r+7-a}{t-r+4-a+b} \ \leq \sum_{\substack{0 \leq b \leq r-3 \\ b \leq t-r+2}} \binom{2t-2r+7}{t-r+4-b} \ \ll \ \frac{4^{t-r} \min{\left(r,\sqrt{t-r+3}\right)}}{\sqrt{t-r+3}} \ . $$ 

Putting everything together, we obtain

$$ \sum_{i \in I} \left(t-|A_i \cap \cup_{k<i} \ A_k| \right) \ \ll \  \frac{4^{t-r} \min{\left(r,\sqrt{t-r+3}\right)}}{\sqrt{t-r+3}} \ + \ (t-r+3)^2 \ + \ t \ \asymp \frac{4^{t-r} \min{\left(r,\sqrt{t-r+3}\right)}}{\sqrt{t-r+3}} + r \ , $$

which proves (\ref{e22}).

\end{proof}

We now use Pikhurko's argument and our improved upper bound for $f(r,t)$ in Theorem \ref{vert bound} to deduce Theorem \ref{cover}.

\begin{repth4}\label{cov}

Every large $K_r$-saturated graph ($r \geq 3$) with $e$ edges has a vertex cover of size $O(e / \log e)$.

\end{repth4}

\begin{proof}

Let $G$ be a $K_r$-saturated graph with $n$ vertices and $e$ edges, where $r \geq 3$ and $n$ is large. If $n<r$, $G=K_n$, so $e/\log e \asymp n^2/\log n$. But $G$ trivially has a vertex cover of size $n = O(n^2/\log n)$, so suppose $n \geq r$. Then $e \geq sat(n,K_r) = (r-2)n-\binom{r-1}{2} \gg r^2$, so $r \ll \sqrt{e}$. \\

Let $t \geq r-2$ be an integer, to be chosen later. Let $A=\{ v \in V(G) : \ d(v) > t \}$ and $B = \{ v \in V(G) : \ d(v) \leq t , \ \exists w \in \Gamma(v) : d(w) \leq t \}$. Then $A \cup B$ is a vertex cover of $G$. We have

$$ |A \cup B| = |A| + |B| \leq \frac{2e}{t} + f(r,t) \ll \frac{e}{t} + \frac{4^{t-r} \ \min{(r,\sqrt{t-r+3})}}{\sqrt{t-r+3}} + r \ll 
\frac{e}{t-r+3} + 4^{t-r+3} + \sqrt{e}  $$

by Theorem \ref{vert bound}. Optimising in $t$ gives a vertex cover of size $O(e / \log e)$. 

\end{proof}

We now use the graphs $G_t$ from section \ref{constr} to prove Theorem \ref{tight}.

\begin{repth5}

Let $r \geq 3$ be an integer and $n$ and $e$ be large quantities with $rn \ll e \ll n^2$. Then there exists a $K_r$-saturated graph $G$ with $|G| \asymp n$ and $e(G) \asymp e$ such that every vertex cover in $G$ has size $\Omega(\min{(e / \log e , n)})$. 

\end{repth5}

\begin{proof}

As when proving the lower bounds for $c(r,t)$ and $f(r,t)$, by adding conical vertices, it is sufficient to prove the special case $r=3$. Indeed, suppose $r$, $n$ and $e$ are as in the theorem. Then by the case $r=3$, there is a $K_3$-saturated graph $H$ with $|H| \asymp n$ and $e(H) \asymp e$ such that every vertex cover in $H$ has size $\Omega(\min{(e / \log e , n)})$. Let $G$ be the graph obtained by adding $r-3$ conical vertices to $H$. Then $G$ is $K_r$-saturated and every vertex cover in $G$ has size $\Omega(\min{(e / \log e , n)})$. Note that $rn \ll e \ll n^2$ implies $r \ll n$ and hence also $r^2 \ll e$. Hence $|G|=|H|+(r-3) \asymp n$ and $e(G)=e(H)+(r-3)|H|+\binom{r-3}{2} \asymp e$. \\

We now prove the special case $r=3$, which states that for all large quantities $n$ and $e$ with $n \ll e \ll n^2$, there exists a $K_3$-saturated graph $G$ with $|G| \asymp n$ and $e(G) \asymp e$ such that every vertex cover in $G$ has size $\Omega(\min{(e / \log e , n)})$. The broad idea of the proof is as follows. Note that by properties 1, 2, 3 and 5 in section \ref{constr}, $G_t$ is $K_3$-saturated, $|G_t| \asymp 4^t t^{-1/2}$, $e(G_t) \asymp 4^t t^{1/2}$ and every vertex cover in $G_t$ has size at least $\frac{1}{2} \binom{2(t-1)}{t-1} \asymp 4^t t^{-1/2}$, respectively. Hence $G_t$ can be used to prove the result when $e / \log e \asymp n$. To extend this to general $n$ and $e$, we will modify $G_t$ depending on whether $e / \log e \ll n$ or  $e / \log e \gg n$. \\

We first consider the case $e / \log e \ll n$. We are given large quantities $n$ and $e$ with $e / \log e \ll n \ll e$ and need to construct a $K_3$-saturated graph $G$ with $|G| \asymp n$ and $e(G) \asymp e$ such that every vertex cover in $G$ has size $\Omega(e / \log e)$. The key to doing this is the fact that $G_t$ has bounded chromatic number. \\

Pick an integer $t \geq 3$ such that $4^t t^{1/2} \asymp e$. Note that $e / \log e \asymp 4^t t^{-1/2} $. By property 6 in section \ref{constr}, there exist three maximally independent subsets of $V(G_t)$ that cover it. Let $G_t''$ be the graph obtained from $G_t$ as follows. Add three new non-adjacent vertices, one for each maximally independent set, and join each new vertex to all vertices in its maximally independent set. Then add one final vertex and join it to the three new vertices. \\

It is easy to check that $G_t''$ has the properties of $G_t$ stated previously - $G_t''$ is $K_3$-saturated, $|G_t''| \asymp 4^t t^{-1/2}$, $e(G_t'') \asymp 4^t t^{1/2}$ and every vertex cover in $G_t''$ has size at least $\frac{1}{2} \binom{2(t-1)}{t-1} \asymp 4^t t^{-1/2}$ - and the additional property of having a vertex of degree three. Pick an integer $N \geq 0$ such that $N \asymp n$. Let $G$ be the graph obtained from $G_t''$ by blowing up the vertex of degree three by $N+1$. Then $G$ is $K_3$-saturated and every vertex cover in $G$ has size $\Omega(e / \log e)$. Note that 

$$|G_t''| \ \asymp \ 4^t t^{-1/2} \ \asymp \ e / \log e \ \ll \ n \ \asymp \ N \ \ll \ e \ \asymp \ 4^t t^{1/2} \ \asymp \ e(G_t'') \ . $$

Hence $|G|=|G_t''|+N \asymp n$ and $e(G)=e(G_t'')+3N \asymp e$. \\

We now consider the case $e / \log e \gg n$. We are given large quantities $n$ and $e$ with $n \log n \ll e \ll n^2$ and need to construct a $K_3$-saturated graph $G$ with $|G| \asymp n$ and $e(G) \asymp e$ such that every vertex cover in $G$ has size $\Omega(n)$. Pick an integer $t \geq 3$ such that $t \asymp e/n$ and $4^t t^{-1/2} \gg n$. This is possible since $e/n \gg \log n$. Note that $t \ll n$. Let $G_t'$ be as in property 7 in section \ref{constr}. \\

Pick an integer $N$ with $0 \leq N \leq \frac{1}{2} \binom{2(t-1)}{t-1}$ and $N \asymp n$. This is possible since $\frac{1}{2} \binom{2(t-1)}{t-1} \asymp 4^t t^{-1/2} \gg n$. Let $G$ be a graph obtained from $G_t'$ by keeping $N$ of the pairs of matching vertices in $R$ and removing the rest. Then $G$ is $K_3$-saturated by property 7 in section \ref{constr}  and contains a matching of size $N$, so every vertex cover in $G$ has size $\Omega(n)$. We have $|G|=2N+(2t-1) \asymp n$ and $e(G)=(2t-1)N+2(t-1) \asymp e$.

\end{proof}

\section{The Two Families Theorem}\label{TFT}

In this section we prove Theorem \ref{estimates}, Lemma \ref{Degenerate Cases} and further related results. We first establish some linear algebra notation in section \ref{linear notation}. Next, we  put our results into context in section \ref{context}. In section \ref{definitions} we introduce some definitions necessary for stating and proving these results, which we then prove in section \ref{results}. 

\subsection{Linear algebra notation}\label{linear notation}

Let $F$ be an \emph{infinite} field. All our vector spaces will be over $F$ and finite dimensional. Given such a vector space $V$, we write $\text{dim}(V)$ for the dimension of $V$. We denote by $0$ the zero subspace of $V$. Given subspaces $V_i$ of $V$, indexed by some finite set $I$, we write $\sum_{i \in I} V_i = \left\{ \sum_{i \in I} v_i : v_i \in V_i \text{ for all } i \right\}$ for the smallest subspace of $V$ contatining all the $V_i$. We say that the sum $\sum_{i \in I} V_i$ is a direct sum if each of its elements can be written \emph{uniquely} as a sum of $v_i \in V_i$. Given a subset $S$ of $V$, we denote by $\text{span}(S)$ the linear span of $S$, the smallest subspace of $V$ containing $S$, consisting of all linear combinations of elements of $S$.

\subsection{Context}\label{context}

The Two Families Theorem is a celebrated result in extremal set theory. It was first stated and proved by Bollob\'as in 1965 (see the lemma in \cite{Reference 9} for the original version of the theorem, which is a weighted generalisation of Theorem \ref{original} below; see also \cite{Reference 10} for an alternative, elegant proof of Theorem \ref{original} by Katona). Since then it has been generalised in several different ways and found numerous applications (see \cite{Reference 11} through \cite{Reference 22}). The simplest version of the Two Families Theorem is as follows. 

\begin{theorem}[Two Families Theorem, Bollob\'as, 1965]\label{original}

Let $a,b \geq 0$ be integers and $(A_i,B_i)_{i \in I}$ be a collection of pairs of finite sets, indexed by a finite set $I$, with the following properties.

\begin{enumerate}
    \item $|A_i| \leq a$ and $|B_i| \leq b$ for all $i \in I$.
    \item $A_i \cap B_i = \emptyset $ for all $i \in I$.
    \item $A_i \cap B_j \neq \emptyset $ for all $i \neq j \in I$.
\end{enumerate}

Then $|I| \leq \binom{a+b}{a}$. 

\end{theorem}

One can see that Theorem \ref{original} is tight by taking $(A_i,B_i)_{i \in I}$ to be the collection of all partitions of a set of size $a+b$ into subsets $A_i$ and $B_i$ of size $a$ and $b$, respectively. Moreover, this is the unique way of achieving equality. An interesting feature of this theorem is that the bound does not depend on the size of the ground set. This is ultimately the reason $c(r,t)$ and $f(r,t)$ are constants independent of $n$. \\

We now recall Theorem \ref{Skew}, which is a generalisation of Theorem \ref{original} allowing one to relax condition 3.

\begin{repth6}[Skew Two Families Theorem, Frankl, 1982]

Let $a,b \geq 0$ be integers and $(A_i,B_i)_{i \in I}$ be a sequence of pairs of finite sets, indexed by a finite, ordered set $I$, with the following properties.

\begin{enumerate}
    \item $|A_i| \leq a$ and $|B_i| \leq b$ for all $i \in I$.
    \item $A_i \cap B_i = \emptyset $ for all $i \in I$.
    \item $A_i \cap B_j \neq \emptyset $ for all $i < j \in I$.
\end{enumerate}

Then $|I| \leq \binom{a+b}{a}$.

\end{repth6}

In this version of the theorem there are many ways of achieving equality. There is no known combinatorial proof of the Skew Two Families Theorem - all known proofs use linear algebra in some form. One can deduce Theorem \ref{Skew} from the following vector space analogue (see Theorem 4.9 in \cite{Reference 23}).

\begin{theorem}[Vector Space Two Families Theorem, Lov\'asz, 1977]\label{Vector}

Let $a,b \geq 0$ be integers and $(A_i,B_i)_{i \in I}$ be a sequence of pairs of finite dimensional vector spaces over $F$, indexed by a finite, ordered set $I$, with the following properties.

\begin{enumerate}
    \item $dim(A_i) \leq a$ and $dim(B_i) \leq b$ for all $i \in I$.
    \item $A_i \cap B_i = 0 $ for all $i \in I$.
    \item $A_i \cap B_j \neq 0 $ for all $i < j \in I$.
\end{enumerate}

Then $|I| \leq \binom{a+b}{a}$.

\end{theorem}

Theorem \ref{Vector} is proved using exterior algebra. We now explain how to deduce Theorem \ref{Skew} from Theorem \ref{Vector}, since the same construction will be useful for us later. Suppose $(A_i,B_i)_{i \in I}$ satisfies the conditions in Theorem \ref{Skew} and has ground set $S$. Let $V$ be a vector space over $F$ with a basis consisting of vectors $e_s$ indexed by $s \in S$. Let $A'_i = \text{span}(\{e_a : a \in A_i\})$ and $B'_i = \text{span}(\{e_b : b \in B_i\})$ for each $i \in I$. Then $(A'_i,B'_i)_{i \in I}$ satisfies the conditions in Theorem \ref{Vector}, so $|I| \leq \binom{a+b}{a}$. We will refer to this construction as \emph{the vector space construction}. \\

Theorem \ref{TF} is similar to the following generalisation of Theorem \ref{Skew} (see \cite{Reference 24}), but with the bound $|A_i| \leq a$ replaced by the weaker bound $|A_i \cap \cup_{k<i} \ A_k| \leq a$.

\begin{theorem}[Threshold Two Families Theorem, Füredi, 1984]\label{t}

Let $a,b \geq c \geq 0$ be integers and $(A_i,B_i)_{i \in I}$ be a sequence of pairs of finite sets, indexed by a finite, ordered set $I$, with the following properties.

\begin{enumerate}
    \item $|A_i| \leq a$ and $|B_i| \leq b$ for all $i \in I$.
    \item $|A_i \cap B_i| \leq c  $ for all $i \in I$.
    \item $|A_i \cap B_j| > c $ for all $i < j \in I$.
\end{enumerate}

Then $|I| \leq \binom{a+b-2c}{a-c}$.

\end{theorem}

Taking $c=0$ gives Theorem \ref{Skew}. One can see that Theorem \ref{t} is tight as follows. Let $(A_i,B_i)_{i \in I}$ be the collection of all partitions of a set $S$ of size $a+b-2c$ into subsets $A_i$ and $B_i$ of size $a-c$ and $b-c$, respectively. Order $I$ arbitrarily and let $T$ be a set of size $c$ disjoint from $S$. Now consider the the sequence $(A'_i,B'_i)_{i \in I}$, where $A'_i = A_i \cup T$ and $B'_i=B_i \cup T$ for each $i \in I$. \\

Füredi deduced Theorem \ref{t} from the following vector space analogue using the vector space construction (see \cite{Reference 24}).

\begin{theorem}[Vector Space Threshold Two Families Theorem, Füredi, 1984]\label{t Vector}

Let $a,b \geq c \geq  0$ be integers and $(A_i,B_i)_{i \in I}$ be a sequence of pairs of finite dimensional vector spaces over $F$, indexed by a finite, ordered set $I$, with the following properties.

\begin{enumerate}
    \item $dim(A_i) \leq a$ and $dim(B_i) \leq b$ for all $i \in I$.
    \item $dim(A_i \cap B_i) \leq c $ for all $i \in I$.
    \item $dim(A_i \cap B_j) > c $ for all $i < j \in I$.
\end{enumerate}

Then $|I| \leq \binom{a+b-2c}{a-c}$.

\end{theorem}

Note that taking $c=0$ gives Theorem \ref{Vector}. Füredi proved Theorem \ref{t Vector} by reducing it to Theorem \ref{Vector}. Since the same argument will be useful later, we explain how this reduction works here. We first need a linear algebra fact. For any two subspaces $U$ and $W$ of a finite dimensional vector space $V$ we have $\text{dim}(U \cap W) \geq \max{(\text{dim}(U)+ \text{dim}(W) -\text{dim}(V),0)}$. We say that $U$ and $W$ are in \emph{general position} if we have equality. The fact we will use is the following. Suppose we have a \emph{finite} collection of subspaces of a finite dimensional vector space $V$ over an \emph{infinite} field. Then for every integer $0 \leq d \leq \text{dim}(V)$, there exists a subspace $W$ of $V$ with $\text{dim}(W) = d$ that is in general position with all subspaces in the collection. \\
 
Now suppose $(A_i,B_i)_{i \in I}$ satisfies the conditions in Theorem \ref{t Vector}. Let $V$ be an ambient vector space containing all the $A_i$ and $B_i$. We may assume that $V$ is finite dimensional and that $\text{dim}(V) \geq c$. Pick a subspace $W$ of $V$ with $\text{dim}(W) = \text{dim}(V) - c$ that is in general position with $A_i$, $B_i$ and $A_i \cap B_i$ for all $i \in I$. Let $A'_i = A_i \cap W$ and $B'_i = B_i \cap W$ for each $i \in I$.
Then $(A'_i,B'_i)_{i \in I}$ satisfies the conditions in Theorem \ref{Vector}, with $a$ and $b$ replaced by $a-c$ and $b-c$, respectively, so $|I| \leq \binom{a+b-2c}{a-c}$. We will refer to this argument as \emph{Füredi's reduction argument}.

\subsection{Definitions}\label{definitions}

In this section we introduce some definitions necessary for stating and proving our results. We first recall Definitions \ref{set} and \ref{mod set}.

\begin{repdef1}

For integers $a,b \geq 0$, an $(a,b)$ \emph{set system} is a sequence $(A_i,B_i)_{i \in I}$ of pairs of finite sets, indexed by a finite, ordered set $I$, with the following properties.

\begin{enumerate}
    \item  $|A_i| \leq a$ and $|B_i| \leq b$ for all $i \in I$.
    \item $A_i \cap B_i = \emptyset  $ for all $i \in I$.
    \item $A_i \cap B_j \neq \emptyset $ for all $i < j \in I$.
\end{enumerate}

\end{repdef1}

\begin{repdef2}

For integers $a,b,c \geq 0$, an $(a,b,c)$ \emph{modified set system} is a sequence $(A_i,B_i)_{i \in I}$ of pairs of finite sets, indexed by a finite, ordered set $I$, with the following properties.

\begin{enumerate}
    \item  $|A_i \cap \cup_{k<i} \ A_k| \leq a$ and $|B_i| \leq b$ for all $i \in I$.
    \item $|A_i \cap B_i| \leq c $ for all $i \in I$.
    \item $|A_i \cap B_j| > c $ for all $i < j \in I$.
\end{enumerate}

\end{repdef2}

We will also need the vector space analogues of these definitions.

\begin{definition}\label{4}

For integers $a,b \geq 0$, an $(a,b)$ \emph{vector space system} is a sequence $(A_i,B_i)_{i \in I}$ of pairs of finite dimensional vector spaces over $F$, indexed by a finite, ordered set $I$, with the following properties.

\begin{enumerate}
    \item $\text{dim}(A_i) \leq a$ and $\text{dim}(B_i) \leq b$ for all $i \in I$.
    \item $A_i \cap B_i = 0 $ for all $i \in I$.
    \item $A_i \cap B_j \neq 0 $ for all $i < j \in I$.
\end{enumerate}

\end{definition}

Note that these are the conditions in Theorem \ref{Vector}.

\begin{definition}\label{mod vec sys}

For integers $a,b,c \geq 0$, an $(a,b,c)$ \emph{modified vector space system} is a sequence $(A_i,B_i)_{i \in I}$ of pairs of finite dimensional vector spaces over $F$, indexed by a finite, ordered set $I$, with the following properties.

\begin{enumerate}
    \item $\text{dim}(A_i \cap \sum_{k < i} A_k) \leq a$ and $\text{dim}(B_i) \leq b$ for all $i \in I$.
    \item $\text{dim}(A_i \cap B_i) \leq c $ for all $i \in I$.
    \item $\text{dim}(A_i \cap B_j) > c $ for all $i < j \in I$.
\end{enumerate}

\end{definition}

We now name the maximum possible size of the index set in a modified set system.

\begin{definition}\label{5}

For all integers $a,b,c \geq 0$, let $i_s(a,b,c)$ be the maximum of $|I|$ over all $(a,b,c)$ modified set systems. 

\end{definition}

We will also need the vector space analogue of this definition.

\begin{definition}\label{6}

For all integers $a,b,c \geq 0$, let $i_v(a,b,c)$ be the maximum of $|I|$ over all $(a,b,c)$ modified vector space systems. 

\end{definition}

We will see in section \ref{results} that these maxima do exist, i.e. that $|I|$ is bounded for both $(a,b,c)$ modified set systems and $(a,b,c)$ modified vector space systems. \\

To prove our results, we will need to also consider the following maximum.

\begin{definition}\label{7}

For all integers $a,b \geq 0$, let $u_s(a,b)$ be the maximum of $|\cup_{i \in I} B_i|$ over all $(a,b)$ set systems. 

\end{definition}

We will also need the vector space analogue of this definition.

\begin{definition}\label{8}

For all integers $a,b \geq 0$, let $u_v(a,b)$ be the maximum of $ \text{dim}\left(\sum_{i \in I} B_i\right)$ over all $(a,b)$ vector space systems. 

\end{definition}

It is clear that these maxima do exist, i.e. that $|\cup_{i \in I} B_i|$ and $\text{dim}(\sum_{i \in I} B_i)$ are bounded, since Theorems \ref{Skew} and \ref{Vector} give $|I| \leq \binom{a+b}{a}$ and condition 1 in Definitions \ref{set} and \ref{4} gives $|B_i| \leq b$ and $\text{dim}(B_i) \leq b$ for $(a,b)$ set and vector space systems, respectively, so $u_s(a,b), u_v(a,b) \leq b \ \binom{a+b}{a} $. We will obtain better bounds, however. \\

We can now state our results.

\begin{theorem}\label{better estimates}

We have the following estimates for $i_s$, $i_v$, $u_s$ and $u_v$.

\begin{enumerate}

\item For all integers $a,b \geq c \geq 0$, 

$$ i_s(a,b,c), \ i_v(a,b,c) \ \asymp \ \binom{a+b-2c+1}{a-c+1} \ . $$ 

\item For all integers $a \geq 0$ and $b \geq 1$, 

$$ u_s(a,b), \ u_v(a,b) \ \asymp \ \binom{a+b+1}{a+1} \ . $$

\item For all integers $b,c \geq a \geq 0$,

$$ i_s(a,b,c) \  = \ i_v(a,b,c) \ = \ \left\lfloor \frac{b-a}{c-a+1} \right\rfloor + 1 \ . $$

\item For all integers $a \geq 0$ and $c \geq b \geq 0$, $i_s(a,b,c)=i_v(a,b,c)=1$. 

\item For all integers $a \geq 0$, $u_s(a,0)=u_v(a,0)=0$. 

\end{enumerate}

\end{theorem}

\newtheorem*{repth7}{Theorem \ref{better estimates}}

Note that Theorem \ref{better estimates} implies Theorem \ref{estimates} and moreover that Theorem \ref{estimates} is tight.

\subsection{Proofs}\label{results}

In this section we prove Theorem \ref{better estimates} and Lemma \ref{Degenerate Cases}. We first prove some preliminary lemmas. Using the vector space construction, one obtains the following lemma.

\begin{lemma}\label{f1<f2}

For all integers $a,b,c \geq 0$, $i_s(a,b,c) \leq i_v(a,b,c)$. For all integers $a,b \geq 0$, $u_s(a,b) \leq u_v(a,b)$.

\end{lemma}

By Lemma \ref{f1<f2}, to prove parts 1, 2 and 3 of Theorem \ref{better estimates}, it will be sufficient to prove lower bounds for $i_s(a,b,c)$ and $u_s(a,b)$ and upper bounds for $i_v(a,b,c)$ and $u_v(a,b)$. \\

Recall that we obtained an extremal construction for Theorem \ref{t} from the extremal construction for Theorem \ref{original} by adjoining $c$ new elements to all the sets. Similarly, adjoining $d$ new elements to all the sets gives $i_s(a+d,b+d,c+d) \geq i_s(a,b,c)$ and the vector space analogue of this argument gives $i_v(a+d,b+d,c+d) \geq i_v(a,b,c)$. Moreover, Füredi's reduction argument shows that the second inequality is in fact an equality.  

\begin{lemma}\label{tred}

For all integers $a,b,c,d \geq 0$, $i_s(a+d,b+d,c+d) \geq i_s(a,b,c)$ and $i_v(a+d,b+d,c+d) = i_v(a,b,c)$.

\end{lemma}

\begin{proof}

Let $(A_i,B_i)_{i \in I}$ be an $(a,b,c)$ modified set or vector space system with ground set $S$ or ambient vector space $V$, respectively. Let $T$ or $W$ be a set or vector space over $F$ of size or dimension $d$ disjoint from $S$ or intersecting $V$ in the zero subspace, respectively. For each $i \in I$, let $A'_i$ be $A_i \cup T$ or $A_i+W$ and $B'_i$ be $B_i \cup T$ or $B_i+W$, respectively. Then $(A'_i,B'_i)_{i \in I}$ is an $(a+d,b+d,c+d)$ modified set or vector space system, respectively, so $i_s(a+d,b+d,c+d) \geq i_s(a,b,c)$ and $i_v(a+d,b+d,c+d) \geq i_v(a,b,c)$.  \\

It remains to show that $i_v(a+d,b+d,c+d) \leq i_v(a,b,c)$. Let $(A_i,B_i)_{i \in I}$ be an $(a+d,b+d,c+d)$ modified vector space system with ambient vector space $V$. We may assume that $V$ is finite dimensional and that $\text{dim}(V) \geq d$. Pick a subspace $W$ of $V$ with $\text{dim}(W) = \text{dim}(V) - d$ that is in general position with $A_i \cap \sum_{k < i} A_k$, $B_i$ and $A_i \cap B_i$ for all $i \in I$. Let $A'_i = A_i \cap W$ and $B'_i = B_i \cap W$ for each $i \in I$. Then $(A'_i,B'_i)_{i \in I}$ is an $(a,b,c)$ modified vector space system, so the inequality follows. 

\end{proof}

Lemma \ref{tred} will allow us to only consider $i_s(a,b,0)$ and $i_v(a,b,0)$ when proving part 1 of Theorem \ref{better estimates} and $i_s(0,b,c)$ and $i_v(0,b,c)$ when proving part 3 of Theorem \ref{better estimates}. \\

We now prove two lemmas. The first, Lemma \ref{f2<f4}, gives an upper bound for $i_v(a,b,0)$ in terms of $u_v(a,b)$. The second, Lemma \ref{f4<f2}, gives an upper bound for $u_v(a,b)$ in terms of the values of $i_v(c,a,0)$ for integers $0 \leq c \leq b-1$. Combining these two will allow us to prove upper bounds for both $i_v(a,b,0)$ and $u_v(a,b)$ by induction. This is the reason for introducing $u_v(a,b)$ in the first place. \\

The main idea in the proof of Lemma \ref{f2<f4} is that, given an $(a,b,0)$ modified vector space system, we would like to replace each $A_i$ by $A_i \cap \sum_{k<i} A_k$, since condition 1 in Definition \ref{mod vec sys} gives an upper bound for $\text{dim}\left(A_i \cap \sum_{k<i} A_k\right)$ rather than $\text{dim}(A_i)$, and then apply Theorem \ref{Vector}. However, this makes the $A_i$ smaller, so while condition 2 in Definition \ref{mod vec sys} will still hold, condition 3 might not. The solution is to find a subset $J \subseteq I$ in which condition 3 does still hold that is maximal with respect to this property in a certain way. We can then bound $|J|$ using Theorem \ref{Vector} and $|I \setminus J|$ using the maximality of $J$.

\begin{lemma}\label{f2<f4}

For all integers $a,b \geq 0$, $i_v(a,b,0) \leq \binom{a+b}{a} + u_v(a,b)$.

\end{lemma}

\begin{proof}

Let $(A_i,B_i)_{i \in I}$ be an $(a,b,0)$ modified vector space system. Define a subset $J \subseteq I $ recursively (in the dual order) as follows. For every $i \in I$, $i \in J$ if and only if $ A_i \cap \left(\sum_{k < i } A_k \right) \cap B_j \neq 0  $ for all $j \in J$ with $i < j$. The set $J$ has the following two properties:

\begin{enumerate}
    \item For all $i < j \in J$, $A_i \cap \left(\sum_{k < i } A_k \right) \cap B_j \neq 0\ $. 
    \item For all $i \in I \setminus J$, there is some $j \in J$ with $i < j$ and $A_i \cap \left( \sum_{k < i } A_k \right) \cap B_j = 0 \ $. 
    
\end{enumerate}

(Indeed, $J$ is the unique subset of $I$ with these properties.) Let $A'_i = A_i \cap \sum_{k < i} A_k $ for each $i \in J$. Order $J$ using the order induced by $I$. Then by property 1, $(A'_i,B_i)_{i \in J}$ is an $(a,b)$ vector space system, so $|J| \leq \binom{a+b}{a}$ by Theorem \ref{Vector} and $\text{dim}\left(\sum_{k \in J} B_k\right) \leq u_v(a,b)$ by definition. \\

For all $i \in I \setminus J$, we clearly have

$$ \text{dim}\left[\left(\sum_{\substack{k \in I \setminus J \\  k \leq i}} A_k \right) \cap \left( \sum_{k \in J} B_k \right) \right] \geq \text{dim}\left[\left(\sum_{\substack{k \in I \setminus J \\  k < i}} A_k \right) \cap \left( \sum_{k \in J} B_k \right) \right] \ . $$

We now show that this inequality is strict. By property 2, there is some $j \in J$ with $i < j$ and $A_i \cap \left(\sum_{k < i } A_k\right) \cap B_j = 0 $. By condition 3 in Definition \ref{mod vec sys}, $A_i \cap B_j \neq 0 $, so there is some $v \in A_i \cap B_j$ with $ v \neq 0$. Then the inequality is strict because $v$ is in the vector space on the left hand side but not in the vector space on the right hand side. We thus have \\

$$ u_v(a,b) \geq \text{dim}\left(\sum_{k \in J} B_k \right) \geq \text{dim}\left[\left(\sum_{i \in I \setminus J} A_i \right) \cap \left( \sum_{k \in J} B_k \right) \right] $$

$$ = \sum_{i \in I \setminus J} \left( \text{dim}\left[\left(\sum_{\substack{k \in I \setminus J \\  k \leq i}} A_k \right) \cap \left( \sum_{k \in J} B_k \right) \right] - \text{dim}\left[\left(\sum_{\substack{k \in I \setminus J \\  k < i}} A_k \right) \cap \left( \sum_{k \in J} B_k \right) \right] \right) \geq |I \setminus J| \ . $$

Hence $|I| = |J| + |I \setminus J| \leq \binom{a+b}{a} + u_v(a,b)$.

\end{proof}

Recall that $u_v(a,b) \leq b \ \binom{a+b}{a}$, since for all $(a,b)$ vector space systems we have $|I| \leq \binom{a+b}{a}$ by Theorem \ref{Vector} and $\text{dim}(B_i) \leq b$ for all $i \in I$ by condition 1 in Definition \ref{4}, so trivially $\text{dim}\left(\sum_{i \in I} B_i \right) \leq b \ \binom{a+b}{a}$. Lemma \ref{f4<f2} gives a better bound in terms of the values of $i_v(c,a,0)$ for integers $0 \leq c \leq b-1$. The main idea in the proof of Lemma \ref{f4<f2} is that for the trivial bound to be close to tight, the sum $\sum_{i \in I} B_i$ would have to be close to a direct sum. But then the $B_i \cap \sum_{k > i} B_k$ would have to have small dimension, which allows us to obtain a better bound for $|I|$ in terms of the values of $i_v(c,a,0)$ for integers $0 \leq c \leq b-1$.

\begin{lemma}\label{f4<f2}
For all integers $a,b \geq 0$,

$$ u_v(a,b) \leq \sum_{c=0}^{b-1} \ \min{\left\{i_v(c,a,0), \binom{a+b}{a} \right\}} \ . $$

\end{lemma}

\begin{proof}

Let $(A_i,B_i)_{i \in I}$ be an $(a,b)$ vector space system. By Theorem \ref{Vector}, $|I| \leq \binom{a+b}{a}$. For each integer $1 \leq c \leq b$, let $I_c = \{ i \in I : \ \text{dim}(B_i) - \text{dim}\left(B_i \cap \sum_{k>i} B_k \right)  \geq c \}$. Then by double counting the number of pairs $(i,c)$, where $i \in I$ and $c$ is an integer, with $1 \leq c \leq \text{dim}(B_i) - \text{dim}\left(B_i \cap \sum_{k > i} B_k \right) $, we obtain

$$ \text{dim}\left(\sum_{i \in I} B_i\right) = \ \sum_{i \in I} \left[ \text{dim}\left(\sum_{k \geq i} B_k \right) - \text{dim} \left(\sum_{k > i} B_k \right) \right] = \ \sum_{i \in I} \left[  \text{dim}(B_i) - \text{dim}\left(B_i \cap \sum_{k > i} B_k \right)  \right] = \ \sum_{c=1}^b |I_c| \ .$$

For each integer $1 \leq c \leq b$, we can bound $|I_c|$ in two different ways. On the one hand, $|I_c| \leq |I| \leq \binom{a+b}{a} $. On the other hand, for all $i \in I_c$, $ \text{dim}\left(B_i \cap \sum_{k > i} B_k \right) \leq b-c$. Order $I_c$ using the order induced by the dual order on $I$. Then $(B_i,A_i)_{i \in I_c}$ is a $(b-c,a,0)$ modified vector space system, so $|I_c| \leq i_v(b-c,a,0)$ by definition. Hence $|I_c| \leq \min{\left\{i_v(b-c,a,0),\binom{a+b}{a} \right\}}$, so

$$ \text{dim}\left(\sum_{i \in I} B_i\right) = \ \sum_{c=1}^b |I_c| \ \leq \ \sum_{c=1}^b \min{\left\{i_v(b-c,a,0),\binom{a+b}{a} \right\}} \ = \ \sum_{c=0}^{b-1} \ \min{\left\{i_v(c,a,0), \binom{a+b}{a} \right\}} \ . $$

\end{proof}

At this point we have all the lemmas we need to prove the upper bounds in parts 1 and 2 of Theorem \ref{better estimates}. To prove the lower bounds, we will need the next two lemmas. The first, Lemma \ref{Product Argument}, gives a lower bound for $i_s(a+c,b+d,0)$ and $i_v(a+c,b+d,0)$ in terms of $i_s(c,d,0)$ and $i_v(c,d,0)$, respectively. We will only need the inequality for $i_s$ and when $c=0$, but proving the general case is not more difficult. 

\begin{lemma}\label{Product Argument}

For all integers $a,b,c,d \geq 0$, $i_s(a+c,b+d,0) \geq \binom{a+b}{a} i_s(c,d,0)$ and $i_v(a+c,b+d,0) \geq \binom{a+b}{a} i_v(c,d,0)$.

\end{lemma}

\begin{proof}

Let $(A_i,B_i)_{i \in I_\ast}$ be an $(a,b)$ set or vector space system with ground set $S_\ast$ or ambient vector space $V_\ast$ and for each $i \in I_\ast$, let $(A_j^i,B_j^i)_{j \in I_i}$ be a $(c,d,0)$ modified set or vector space system with ground set $S_i$ or ambient vector space $V_i$, respectively. Without loss of generality, we may assume that the $I_i$ are disjoint and that the union $\cup_{i \in I \cup \{\ast\}} \ S_i$ is a disjoint union or that the sum $\sum_{i \in I \cup \{\ast\}} \ V_i$ is a direct sum, respectively. \\

Let $J= \cup_{i \in I_\ast} \ I_i$, ordered with the sum of the orders on the $I_i$. For each $i \in I_\ast$ and $j \in I_i$, let $A'_j$ be $A_i \cup A^i_j$ or $A_i + A^i_j$ and $B'_j$ be $B_i \cup B^i_j$ or $B_i + B^i_j$, respectively. Then $(A'_i,B'_i)_{i \in J}$ is an $(a+c,b+d,0)$ modified set or vector space system, respectively, so the inequalities follow.  

\end{proof}

Given an $(a,b)$ set or vector space system $(A_i,B_i)_{i \in I}$, the second lemma, Lemma \ref{f3>}, gives a lower bound for $u_s(a+c,b+d)$ or $u_v(a+c,b+d)$ in terms of $u_s(c,d)$ or $u_v(c,d)$, $|I|$ and $|\cup_{i \in I} B_i|$ or $\text{dim}\left(\sum_{i \in I} B_i\right)$, respectively. We will only need the inequality for $u_s$ and when $c=0$ or $d=1$, but proving the general case is not more difficult. The proof is similar to that of Lemma \ref{Product Argument}.

\begin{lemma}\label{f3>}

Let $a,b,c,d \geq 0$ be integers and $(A_i,B_i)_{i \in I}$ be an $(a,b)$ set or vector space system. Then $u_s(a+c,b+d) \geq |I|u_s(c,d) + |\cup_{i \in I} B_i|$ or  $u_v(a+c,b+d) \geq |I|u_v(c,d) + \text{dim}\left(\sum_{i \in I} B_i\right)$, respectively.

\end{lemma}

\begin{proof}

Let $I_\ast=I$ and let $(A_i,B_i)_{i \in I_\ast}$ have ground set $S_\ast$ or ambient vector space $V_\ast$, respectively. For each $i \in I_\ast$, let $(A_j^i,B_j^i)_{j \in I_i}$ be a $(c,d)$ set or vector space system with ground set $S_i$ or ambient vector space $V_i$, respectively. As before, we may assume without loss of generality that the $I_i$ are disjoint and that the union $\cup_{i \in I \cup \{\ast\}} \ S_i$ is a disjoint union or that the sum $\sum_{i \in I \cup \{\ast\}} \ V_i$ is a direct sum, respectively. Define and order $J$ as before and for each $i \in J$, define $A'_i$ and $B'_i$ as before. Then $(A'_i,B'_i)_{i \in J}$ is an $(a+c,b+d)$ set or vector space system, respectively, so the inequalities follow.

\end{proof}

We now prove the vector space analogue of Lemma \ref{Degenerate Cases}.

\begin{lemma}\label{deg vec}

Let $b,c \geq a \geq 0$ be integers and $((A_i)_{i \in I}, B)$ be a pair, where $(A_i)_{i \in I}$ is a sequence of finite dimensional vector spaces over $F$, indexed by a finite, ordered set $I$, and $B$ is a finite dimensional vector space over $F$, with the following properties. 

\begin{enumerate}
    \item $\text{dim}(B) \leq b$.
    \item $\text{dim}(A_i \cap B) > c $ for all $i \in I$.
    \item $\text{dim}\left(A_i \cap \sum_{k<i} A_k \right) \leq a $ for all $i \in I$.
\end{enumerate}

Then $|I| \leq \left\lfloor \frac{b-a}{c-a+1} \right\rfloor$.

\end{lemma}

\begin{proof}

If $I$ is empty, the bound holds, so suppose otherwise. Then by condition 1,

$$ b \geq \text{dim}(B)\geq \text{dim}\left(\sum_{i \in I} (A_i \cap B) \right) = \sum_{i \in I} \left[ \text{dim}\left(\sum_{k \leq i} (A_k \cap B) \right) - \text{dim}\left(\sum_{k < i} (A_k \cap B) \right) \right] $$

$$ = \sum_{i \in I} \left[ \text{dim}\left(A_i \cap B\right) - \text{dim}\left(A_i \cap \sum_{k < i} (A_k \cap B) \right) \right] \geq \sum_{i \in I}  \left[ \text{dim}\left(A_i \cap B\right) - \text{dim}\left(A_i \cap \sum_{k < i} A_k \right) \right] . $$

But by conditions 2 and 3,

\begin{equation}\label{e999}
\text{dim}\left(A_i \cap B\right) - \text{dim}\left(A_i \cap \sum_{k < i} A_k \right) \geq c-a+1
\end{equation}

for all $i \in I$. Moreover, the left hand side of (\ref{e999}) is at least $c+1$ for the first element of $I$, so $b \geq (c-a+1)|I|+a$ and the bound follows.

\end{proof}

\begin{rem}

Alternatively, one can use Füredi's reduction argument to reduce Lemma \ref{deg vec} to the special case $a=0$, which is easier to prove.

\end{rem}

Lemma \ref{Degenerate Cases} follows from Lemma \ref{deg vec} using the vector space construction. (Alternatively, one can prove Lemma \ref{Degenerate Cases} directly using the set analogue of the proof of Lemma \ref{deg vec}.) \\

We are now ready to prove Theorem \ref{better estimates}. 

\begin{repth7}

We have the following estimates for $i_s$, $i_v$, $u_s$ and $u_v$.

\begin{enumerate}

\item For all integers $a,b \geq c \geq 0$, 

$$ i_s(a,b,c), \ i_v(a,b,c) \ \asymp \ \binom{a+b-2c+1}{a-c+1} \ . $$ 

\item For all integers $a \geq 0$ and $b \geq 1$, 

$$ u_s(a,b), \ u_v(a,b) \ \asymp \ \binom{a+b+1}{a+1} \ . $$

\item For all integers $b,c \geq a \geq 0$,

$$ i_s(a,b,c) \ = \ i_v(a,b,c) \ = \ \left\lfloor \frac{b-a}{c-a+1} \right\rfloor + 1 \ . $$

\item For all integers $a \geq 0$ and $c \geq b \geq 0$, $i_s(a,b,c)=i_v(a,b,c)=1$. 

\item For all integers $a \geq 0$, $u_s(a,0)=u_v(a,0)=0$. 

\end{enumerate}

\end{repth7}

\begin{proof}

To prove part 1 it suffices to show $i_s(a,b,c) \gg \binom{a+b-2c+1}{a-c+1}$ and $i_v(a,b,c) \ll \binom{a+b-2c+1}{a-c+1}$ by Lemma \ref{f1<f2}. By Lemma \ref{tred}, it suffices to prove the special case $c=0$, which states that $i_s(a,b,0) \gg \binom{a+b+1}{a+1}$ and $i_v(a,b,0) \ll \binom{a+b+1}{a+1}$ for  all integers $a,b \geq 0$. Let us first prove the lower  bound. By Lemma \ref{Product Argument} with $c=0$ and part 3 with $a=0=c$, we have $i_s(a,b,0) \geq \max_{0 \leq c \leq b} \ (c+1) \binom{a+b-c}{a} \asymp \binom{a+b+1}{a+1}$. The maximum is attained at $c = \left \lceil \frac{b-a}{a+1} \right \rceil$. \\

Let us now prove the upper bound. We will show that $i_v(a,b,0) \leq 2 \binom{a+b+1}{a+1}$ and $u_v(a,b) \leq 2 \binom{a+b+1}{a+1}-\binom{a+b}{a}$ for  all integers $a,b \geq 0$ by induction on $a+b$. By induction, we may assume $i_v(c,a,0) \leq 2 \binom{a+c+1}{a}$ for all integers $0 \leq c \leq b-1$. Then

$$ u_v(a,b) \leq \sum_{c=0}^{b-1} \min{\left\{2 \binom{a+c+1}{a},\binom{a+b}{a} \right\}} \leq 2 \sum_{c=0}^b \binom{a+c}{a} - \binom{a+b}{a} = 2 \binom{a+b+1}{a+1} - \binom{a+b}{a} \  $$

by Lemma \ref{f4<f2} and $i_v(a,b,0) \leq 2 \binom{a+b+1}{a+1}$ by Lemma \ref{f2<f4}. \\

To prove part 2 it suffices to show $u_s(a,b) \gg \binom{a+b+1}{a+1}$ and $u_v(a,b) \ll\binom{a+b+1}{a+1}$ by Lemma \ref{f1<f2}. We have already proved the upper bound, so let us prove the lower bound. By Lemma \ref{f3>}, $u_s(a,b) \geq |I|u_s(c,d)+|\cup_{i \in I} B_i| \geq |I|u_s(c,d)$ for all integers $0 \leq c \leq a$ and $0 \leq d \leq b$ and all $(a-c,b-d)$ set systems $(A_i,B_i)_{i \in I}$, so $u_s(a,b) \geq \binom{a+b-c-d}{a-c} u_s(c,d)$ for all  integers $0 \leq c \leq a$ and $0 \leq d \leq b$. \\

Taking $c=0$ and using the fact that $u_s(0,d)=d$ for all integers $d \geq 0$, we obtain $u_s(a,b) \geq \max_{0 \leq d \leq b} \ d \binom{a+b-d}{a} \asymp \binom{a+b}{a+1}$. The maximum is attained at $d = \left \lceil \frac{b}{a+1} \right \rceil$. Taking $d=1$ and using the fact that $u_s(c,1)=c+1$ for all integers $c \geq 0$, we obtain $u_s(a,b) \geq \max_{0 \leq c \leq a} \ (c+1) \binom{a+b-1-c}{b-1} \asymp \binom{a+b}{a}$. The maximum is attained at $c = \left \lceil \frac{a-b+1}{b} \right \rceil$. Combining the two bounds, we obtain $u_s(a,b) \gg \binom{a+b}{a+1} +  \binom{a+b}{a} = \binom{a+b+1}{a+1}$.   \\

To prove part 3 it suffices to show $i_s(a,b,c) \geq  \left\lfloor \frac{b-a}{c-a+1} \right\rfloor + 1 $ and $i_v(a,b,c) \leq \left \lfloor \frac{b-a}{c-a+1} \right \rfloor + 1$ by Lemma \ref{f1<f2}. Let us first prove the lower bound. By Lemma \ref{tred}, it suffices to prove the special case $a=0$, which states that $i_s(0,b,c) \geq  \left\lfloor \frac{b}{c+1} \right\rfloor + 1 $ for all integers $b,c \geq 0$. Let $(A_i)_{i \in I}$ be a sequence of disjoint sets of size $c+1$, indexed by an ordered set $I$ with $|I|=\left \lfloor \frac{b}{c+1} \right \rfloor + 1$. For each $i \in I$, let $B_i = \cup_{k \neq i} \ A_k$. Then it is easy to check that $(A_i,B_i)_{i \in I}$ is a $(0,b,c)$ modified set system. \\

We now prove the upper bound. Let $(A_i,B_i)_{i \in I}$ be an $(a,b,c)$ modified vector space system. If $I$ is empty, the bound holds, so suppose otherwise. Let $l$ be the last element of $I$. Order $I \setminus \{l\}$ using the order induced by $I$. Then $((A_i)_{i \in I \setminus \{l\}},B_l)$ satisfies the conditions in Lemma \ref{deg vec}, so the bound follows. \\

We now prove part 4. By condition 1 in Definition \ref{mod set} or \ref{mod vec sys}, we have $|A_i \cap B_j| \leq |B_j| \leq b \leq c$ or $\text{dim}(A_i \cap B_j) \leq \text{dim}(B_j) \leq b \leq c$ for all $ i < j \in I$ for all $(a,b,c)$ modified set or vector space systems, respectively, so condition 3 in  Definition \ref{mod set} or \ref{mod vec sys} forces $|I| \leq 1$. One can trivially construct $(a,b,c)$ modified set and vector space systems with $|I|=1$, so $i_s(a,b,c)=1=i_v(a,b,c)$. Part 5 is trivial.

\end{proof}

\begin{rem}

Using the set analogues of the proofs of Lemmas \ref{f2<f4} and \ref{f4<f2}, one can prove the set analogues of these lemmas, namely $i_s(a,b,0) \leq \binom{a+b}{a} + u_s(a,b)$ and $u_s(a,b) \leq \sum_{c=0}^{b-1} \ \min{\{i_s(c,a,0),\binom{a+b}{a}\}} $. One can then show $i_s(a,b,0), u_s(a,b) \ll \binom{a+b+1}{a+1}$ by induction just as in the proof of Theorem \ref{better estimates}. However, Lemma \ref{tred} gives only $i_s(a+d,b+d,c+d) \geq i_s(a,b,c)$ as opposed to $i_v(a+d,b+d,c+d) = i_v(a,b,c)$, since there is no set analogue of Füredi's reduction argument. Because of this, we cannot deduce an upper bound for $i_s(a,b,c)$ from the upper bound for $i_s(a,b,0)$ as we did for $i_v$. This is the reason for introducing all the vector space analogues in the first place and then using Lemma \ref{f1<f2} to deduce their set analogues. \\

One could instead generalise Definition \ref{set} to that of an $(a,b,c)$ set system by replacing conditions 2 and 3 by conditions 2 and 3 in Definition \ref{mod set} and then define $u_s(a,b,c)$ as in Definition \ref{7}. Then the set analogues of the proofs of Lemmas \ref{f2<f4} and \ref{f4<f2} can be generalised to prove that $i_s(a,b,c) \leq \binom{a+b-2c}{a-c} + u_s(a,b,c)$ and $u_s(a,b,c) \leq \sum_{d=0}^{b-1} \ \min{\{i_s(d,a,c),\binom{a+b-2c}{a-c}\}} $ for $a,b \geq c$. \\

One can then deduce upper bounds for $i_s(a,b,c)$ and $u_s(a,b,c)$ by induction. The problem with this approach is that, even when proving the upper bounds for $a,b \geq c$, the degenerate cases arise in the induction, leading to inferior bounds. Indeed, this must be the case, since we trivially have $u_s(a,b,c) \geq b$ by taking $|I|=1$, so the upper bounds obtained must increase not only with $a-c$ and $b-c$, but also with $c$.

\end{rem}

\begin{rem}

We needed $F$ to be infinite for Füredi's reduction argument. If one is only interested in $i_v$ and $u_v$ as means to prove the estimates for $i_s$ and $u_s$, one can simply take $F$ to be infinite. However, if one is interested in $i_v$ and $u_v$ in their own right, one might ask whether our estimates still hold when $F$ is finite. This is indeed the case and can be deduced from the case when $F$ is infinite, as follows. \\

Note that Lemma \ref{f1<f2} still holds when $F$ is finite, so our lower bounds for $i_v$ and $u_v$ still hold. Let $F'$ be an infinite field containing $F$. There is an algebraic technique known as extension of scalars which, given a vector space $V$ over $F$, constructs an associated vector space $V'$ over $F'$. It is easy to check that for any $(a,b,c)$ modified vector space system $(A_i,B_i)_{i \in I}$ over $F$, $(A'_i,B'_i)_{i \in I}$ is an $(a,b,c)$ modified vector space system over $F'$ and that for any $(a,b)$ vector space system $(A_i,B_i)_{i \in I}$ over $F$, $(A'_i,B'_i)_{i \in I}$ is an $(a,b)$ vector space system over $F'$ with $\text{dim}\left(\sum_{i \in I} B'_i\right)=\text{dim}\left(\sum_{i \in I} B_i\right)$. Hence the upper bounds for $i_v$ and $u_v$ over $F$ follow from those over $F'$.

\end{rem}

\section{Known values of $c(r,t)$}\label{known}

In this section we first explain how, in principle, given $r$ and $t$, one can determine $c(r,t)$ by a finite computation. We then discuss the known values of $c(r,t)$. Finally, we complete the work of Duffus and Hanson in \cite{Reference 3} on the case $r=3=t$ by determining $C(3,3)$. \\

One can extract explicit lower and upper bounds for $c(r,t)$ and $c'(r,t)$ from the proof of Theorem \ref{c est}. In section \ref{bu}, we saw that $|G| \leq c'(r,t+1)+k$ for all $G \in C'(r,t,k)$, so one can obtain an explicit upper bound for $|G|$ for all $G \in C'(r,t,k)$. One can also extract an explicit upper bound for $|G|$ for all $G \in C(r,t,k)$ from the proof of Theorem \ref{blow-up}. Hence, in principle, given $r$, $t$ and $k$, one can determine $C(r,t,k)$ and $C'(r,t,k)$ by a finite computation. Note that $-c(r,t)$ and $-c'(r,t)$ are the smallest values of $k$ for which $C(r,t,k)$ and $C(r,t,k) \cup C'(r,t,k)$ are not empty, respectively. Hence, in principle, given $r$ and $t$, one can determine $c(r,t)$ and $c'(r,t)$ by a finite computation. In practice, however, even for small values of $r$ and $t$ (and $k$), these computations are unfeasible. \\

For $t \leq r$, the value of $c(r,t)$ is known. To the best of the author's knowledge, no value of $c(r,t)$ is known with $t>r$. We first consider the case $t=r-2$. Recall (see section \ref{intro}) that for all integers $r \geq 2$ and $n \geq r-2$, $sat(n,K_r)=(r-2)n-\binom{r-1}{2}$ and that the unique extremal graph consists of a $K_{r-2}$ fully connected to an independent set of size $n-(r-2)$. Note that (for $n \geq r-1$) this graph has minimum degree $r-2$ and can be obtained from $K_{r-1}$ by blowing up a vertex by $n-(r-2)$. Hence $c(r,r-2)=\binom{r-1}{2}=c'(r,r-2)$, $C(r,r-2)=\{K_{r-1}\}$ and $C'(r,r-2)=\emptyset$ for all integers $r \geq 3$. So the answer to Question \ref{q} (whether $c'(r,t)=c(r,t)$) is ``Yes." when $t=r-2$. \\

Next, we consider the case $t=r-1$. It is easy to show that all $K_3$-saturated graphs $G$ with $\delta(G)=2$ are blow-ups of either the complete bipartite graph $K_{2,2}$ or the cycle $C_5$. Recall (see section \ref{sec con}) that every $K_r$-saturated graph $G$ with $\delta(G)<2(r-2)$ has a conical vertex. In particular, every $K_r$-saturated graph $G$ with $\delta(G)=r-1$ has a conical vertex for all integers $r \geq 4$. Hence, all $K_r$-saturated graphs $G$ with $\delta(G)=r-1$ are blow-ups of either $K_{2,2}^{r-3}$ or $C_5^{r-3}$. So $c(r,r-1)=\binom{r}{2}+2$ and $C(r,r-1)=\{C_5^{r-3}\}$ for all integers $r \geq 3$. \\

Finally, we consider the case $t=r$. Duffus and Hanson showed that $c(3,3)=15$ (Theorem  4 in \cite{Reference 3}) and came close to determining $C(3,3)$. Let $P$ be the Petersen graph (we will give a precise description of $P$ later). Duffus and Hanson noted that $P$ is $K_3$-saturated and $\delta(P)=3$. We have $|P|=10$ and $e(P)=15$. Hence $c(3,3) \geq 15$. Duffus and Hanson showed that $c(3,3) \leq 15$ as follows. Let $G$ be a $K_3$-saturated graph with $\delta(G)=3$. We need to show that $e(G) \geq 3|G|-15$. \\

Duffus and Hanson first showed that for such $G$, either $e(G)>3|G|- 15$ or $G \supseteq P$ (Lemma 4.1 in \cite{Reference 3}). They then showed that if $G \supseteq P$, $e(G) \geq 3|G|-15$ (proof of Theorem 4 in \cite{Reference 3}), completing the proof. All that was missing to determine the extremal graphs, or equivalently $C(3,3)$, was to examine when we have equality in this final step of the argument. \\

It is easy to see from Duffus and Hanson's proof that the copy of $P$ they find in $G$ has a vertex of degree $3$ in $G$ ($x$ in the proof of Lemma 4.1 in \cite{Reference 3}), or in other words a vertex whose only neighbours in $G$ are those in $P$. This observation allows us to simplify Duffus and Hanson's proof that $c(3,3) \leq 15$ by replacing the final step of the argument with a proof that if $G \supseteq P$ \emph{and} $P$ contains a vertex of degree $3$ in $G$, $e(G) \geq 3|G|-15$. We then determine $C(3,3)$ by considering when we have equality. \\

Let $s$ be a vertex of $P$. Then $P$ can be described as follows. The vertex set of $P$ is $\{s,x_1,x_2,x_3,y_1,y_2,y_3,z_1,z_2,z_3\}$ and the edges are as follows. $s$ is adjacent to all the $x_i$, each $x_i$ is adjacent to $y_i$ and $z_i$, $y_i$ and $z_j$ are adjacent for all $i \neq j$ and there are no other edges. Let $Q$ be the graph with the same description as $P$, but with the indices $i$ ranging from $1$ to $4$ instead. So $|Q|=13$ and $e(Q)=24$. It is easy to check that $Q$ is $K_3$-saturated and $\delta(Q)=3$. Note that $e(Q)=3|Q|-15$. For both $P$ and $Q$ it is easy to check that all vertices have distinct neighbourhoods. Hence $\{P,Q\} \subseteq C(3,3)$. This shows that, while Theorem \ref{blow-up} guarantees that $C(r,t)$ is always finite, it can contain more than one graph. We now show that this containment is in fact an equality. 

\begin{theorem}\label{P,Q}
$C(3,3)=\{P,Q\}$.
\end{theorem}

\begin{proof}

Let $G$ be a $K_3$-saturated graph with $\delta(G)=3$. Suppose $G \supseteq P$ and $P$ contains a vertex $s$ of degree $3$ in $G$. We need to show that $e(G) \geq 3|G|-15$, with equality only if $G$ is obtained from either $P$ or $Q$ by blowing up non-adjacent vertices of degree $3$. Note that, since $P$ is $K_3$-saturated and $G$ is $K_3$-free, $P$ must be an induced subgraph of $G$.  We will use the previous description of $P$. For each $v \in V(G) \setminus V(P)$, let $d_P(v)=|\Gamma(v)\cap V(P)|$, $d_>(v)=|\{w \in \Gamma(v) \setminus V(P) : \ d_P(w)>d_P(v) \}|$ and $d_=(v)=|\{w \in \Gamma(v) \setminus V(P) : \ d_P(w)=d_P(v) \}|$. We then have

\begin{equation*}
e(G)=15+\sum_{v \in V(G) \setminus V(P)} \left(d_P(v)+d_>(v)+\frac{1}{2}d_=(v) \right) \ .
\end{equation*}

\leavevmode\newline Our aim now is to show that $d_P(v) + d_>(v) + \frac{1}{2}d_=(v) \geq 3$ for all $v \in V(G) \setminus V(P)$, so that $e(G) \geq 15+3(|G|-10)=3|G|-15$. We first prove three claims. \\

\emph{Claim 1:} Every $v \in V(G) \setminus V(P)$ is adjacent to one of the $x_i$. \\

\emph{Proof:} $G$ is $K_3$-saturated and $s$ has degree $3$ in $G$. \\

\emph{Claim 2:} Let $v \in V(G) \setminus V(P)$ be non-adjacent to $y_3$, $z_2$ and all their neighbours in $P$. Then there exist distinct vertices $u,w \in \Gamma(v) \setminus V(P)$ with $d_P(u),d_P(w) \geq 2$ adjacent to $y_3$ and $z_2$, respectively. \\

\emph{Proof:} Since $G$ is $K_3$-saturated, there must exist vertices $u, w \in \Gamma(v) \setminus V(P)$ adjacent to $y_3$ and $z_2$, respectively. Since $y_3$ and $z_2$ are adjacent and $G$ is $K_3$-free, $u$ and $w$ must be distinct. By Claim 1, $u$ and $w$ must also be adjacent to one of the $x_i$, so $d_P(u),d_P(w) \geq 2$.  \\

\emph{Claim 3:} Suppose $v \in V(G) \setminus V(P)$ and $d_P(v)=2$. Then, up to symmetry, $\Gamma(v) \cap V(P) = \{x_1,y_2\}$. \\

\emph{Proof:} By Claim 1, $v$ is adjacent to at least one of the $x_i$. We first show that $v$ is adjacent to exactly one of the $x_i$. Suppose for the sake of contradiction that $v$ is adjacent to two of the $x_i$, say $x_1$ and $x_2$. Then, since $d_P(v)=2$, $v$ is not adjacent to $y_3$ and all of its neighbours in $P$. Hence, since $G$ is $K_3$-saturated, there must be a vertex $w \in V(G) \setminus V(P)$ adjacent to both $v$ and $y_3$. But then, since $x_3$ is adjacent to $y_3$ and $G$ is $K_3$-free, $w$ cannot be adjacent to any of the $x_i$, contradicting Claim 1. \\

So $v$ is adjacent to exactly one of the $x_i$, say $x_1$. Then $v$ is not adjacent to $s$, since $s$ has degree $3$ in $G$, and not adjacent to $y_1$ and $z_1$, since $x_1$ is adjacent to both $y_1$ and $z_1$ and $G$ is $K_3$-free. Hence the other neighbour of $v$ in $P$ must be one of $y_2$, $y_3$, $z_2$ or $z_3$, so up to symmetry, $\Gamma(v) \cap V(P) = \{x_1,y_2\}$. \\

We now show that $d_P(v)+d_>(v)+\frac{1}{2}d_=(v) \geq 3$ for all $v \in V(G) \setminus V(P)$. By Claim 1, $d_P(v) \geq 1$. We consider the three cases $d_P(v) = 1$, $d_P(v) = 2$ and $d_P(v) \geq 3$ separately. \\

\emph{Case 1:} $d_P(v) = 1$. \\

\emph{Proof:} By Claim 1, the unique neighbour of $v$ in $P$ must be one of the $x_i$, say $x_1$. Then by Claim 2, there exist distinct vertices $u,w \in \Gamma(v) \setminus V(P)$ with $d_P(u),d_P(w) \geq 2$. Hence $d_>(v) \geq 2$, so $d_P(v)+d_>(v)+\frac{1}{2}d_=(v) \geq 3$. \\

\emph{Case 2:} $d_P(v)=2$. \\

\emph{Proof:} By Claim 3, without loss of generality, $\Gamma(v) \cap P = \{x_1,y_2\}$. Then by Claim 2, there exist distinct vertices $u,w \in \Gamma(v) \setminus V(P)$ with $d_P(u),d_P(w) \geq 2$. Hence $d_>(v)+d_=(v) \geq 2$, so $d_P(v)+d_>(v)+\frac{1}{2}d_=(v) \geq 3$. \\

\emph{Case 3:} $d_P(v) \geq 3$. \\

\emph{Proof:} It immediately follows that $d_P(v)+d_>(v)+\frac{1}{2}d_=(v) \geq 3$. \\

This concludes the proof that $d_P(v)+d_>(v)+\frac{1}{2}d_=(v) \geq 3$ for all $v \in V(G) \setminus V(P)$ and hence $e(G) \geq 3|G|-15$. Suppose now that $e(G)=3|G|-15$. We need to show that $G$ is obtained from either $P$ or $Q$ by blowing up non-adjacent vertices of degree $3$. We must have $d_P(v)+d_>(v)+\frac{1}{2}d_=(v) = 3$ for all $v \in V(G) \setminus V(P)$, which translates to the following.

\begin{itemize}

\item For $v \in V(G) \setminus V(P)$ with $d_P(v)=1$, we have $d_>(v)=2$ and $d_=(v)=0$.
\item For $v \in V(G) \setminus V(P)$ with $d_P(v)=2$, we have $d_>(v)=0$ and $d_=(v)=2$.
\item For $v \in V(G) \setminus V(P)$ with $d_P(v) \geq 3$, we have $d_P(v)=3$ and $d_=(v)=0$.

\end{itemize}

Our aim now is to show that there do not exist vertices $v \in V(G) \setminus V(P)$ with $d_P(v)=2$. We first prove the following claim. \\

\emph{Claim 4:} Suppose $v \in V(G) \setminus V(P)$ and $\Gamma(v) \cap V(P) = \{x_1,y_2\}$. Then $\{a \in \Gamma(v) \setminus V(P) : \ d_P(a)=2 \}=\{u,w\}$, for some vertices $u$ and $w$ with $\Gamma(u) \cap V(P) = \{x_2,y_3\}$ and $\Gamma(w) \cap V(P) = \{x_3,z_2\}$. \\

\emph{Proof:} By Claim 2, there exist distinct vertices $u,w \in \Gamma(v) \setminus V(P)$ with $d_P(u),d_P(w) \geq 2$ adjacent to $y_3$ and $z_2$, respectively. Since $d_>(v)=0$,  we must have $d_P(u)=2=d_P(w)$. Then since $u$ and $w$ are distinct and $d_=(v)=2$, we must have $\{a \in \Gamma(v) \setminus V(P) : \ d_P(a)=2 \}=\{u,w\}$. Now, $u$ cannot be adjacent to $x_1$ or $x_3$, since $x_3$ and $y_3$ are adjacent and $G$ is $K_3$-free, so $u$ must be adjacent to $x_2$ by Claim 1. Since $d_P(u)=2$, we must have $\Gamma(u) \cap V(P) = \{x_2,y_3\}$. Similarly, $w$ cannot be adjacent to $x_1$ or $x_2$, since $x_2$ and $z_2$ are adjacent and $G$ is $K_3$-free, so $w$ must be adjacent to $x_3$ by Claim 1. Since $d_P(w)=2$, we must have $\Gamma(w) \cap V(P) = \{x_3,z_2\}$. \\

We now show that there do not exist vertices $v \in V(G) \setminus V(P)$ with $d_P(v)=2$. Suppose for the sake of contradiction that there is a vertex $v \in V(G) \setminus V(P)$ with $d_P(v)=2$. By Claim 3, without loss of generality, $\Gamma(v) \cap V(P) = \{x_1,y_2\}$. Then by Claim 4, $\{a \in \Gamma(v) \setminus V(P) : \ d_P(a)=2 \}=\{u,w\}$, for some vertices $u$ and $w$ with $\Gamma(u) \cap V(P) = \{x_2,y_3\}$ and $\Gamma(w) \cap V(P) = \{x_3,z_2\}$. By Claim 4 again and symmetry, $\{a \in \Gamma(u) \setminus V(P) : \ d_P(a)=2 \}=\{b,c\}$, for some vertices $b$ and $c$ with $\Gamma(b) \cap V(P) = \{x_3,y_1\}$ and $\Gamma(c) \cap V(P) = \{x_1,z_3\}$. But then $v \in \{a \in \Gamma(u) \setminus V(P) : d_P(a)=2 \}=\{b,c\}$, which is impossible since $\Gamma(v) \cap V(P)$ is distinct from $\Gamma(b) \cap V(P)$ and $\Gamma(c) \cap V(P)$. \\

We now consider the case where vertices $v \in V(G) \setminus V(P)$ with $d_P(v)=1$ do not exist and the case where vertices $v \in V(G) \setminus V(P)$ with $d_P(v)=1$ do exist separately. These cases correspond to the case where $G$ is obtained from $P$ by blowing up non-adjacent vertices of degree $3$ and the case where $G$ is obtained from $Q$ by blowing up non-adjacent vertices of degree $3$, respectively. \\

\emph{Case 1:} There do not exist vertices $v \in V(G) \setminus V(P)$ with $d_P(v)=1$. \\

\emph{Proof:} We have $d_P(v)=3$ for all $v \in V(G) \setminus V(P)$. Moreover, all these vertices are non-adjacent, since $d_=(v)=0$ for such vertices. Since $G$ is $K_3$-free, the neighbourhoods of these vertices are independent subsets of $V(P)$. Since $G$ is $K_3$-saturated, these sets are in fact maximally independent subsets of $V(P)$. Finally, since the vertices outside $V(P)$ are non-adjacent and $G$ is $K_3$-saturated, these maximally independent subsets of $V(P)$ are also intersecting. \\

It is easy to check that in $P$, all maximally independent subsets of size $3$ are the neighbourhood of a vertex. Hence the neighbourhood of every vertex outside $V(P)$ is the neighbourhood in $P$ of some vertex in $V(P)$. Since these neighbourhoods are intersecting and $P$ is $K_3$-free, these vertices in $V(P)$ are non-adjacent. Hence $G$ is obtained from $P$ by blowing up non-adjacent vertices (of degree $3$). \\

\emph{Case 2:} There do exist vertices $v \in V(G) \setminus V(P)$ with $d_P(v)=1$. \\

\emph{Proof:} We first prove the following claim. \\

\emph{Claim 5:} Suppose $v \in V(G) \setminus V(P)$ and $d_P(v)=1$. Then, up to symmetry, $\Gamma(v)=\{x_1,b,c\}$, for some vertices $b,c \in V(G) \setminus V(P)$ with $\Gamma(b) \cap V(P)=\{x_2,y_3,z_3\}$ and $\Gamma(c) \cap V(P)=\{x_3,y_2,z_2\}$. Moreover, every vertex $w \in V(G) \setminus V(P)$ with $d_P(w)=3$ is adjacent to $x_1$ if and only if $w \not \in \{b,c\}$. \\

\emph{Proof:} By Claim 1, without loss of generality, $\Gamma(v) \cap V(P) = \{x_1\}$. Since $d_>(v)=2$ and $d_=(v)=0$, $\Gamma(v) \setminus V(P) = \{b,c\}$, for some vertices $b$ and $c$ with $d_P(b)=3=d_P(c)$. Let $S=\{x_2,x_3,y_2,y_3,z_2,z_3\}$. Then $v$ is not adjacent to every vertex in $S$ and every neighbour in $P$ of a vertex in $S$ and $G$ is $K_3$-saturated, so $S \subseteq (\Gamma(b) \cap V(P)) \cup (\Gamma(c) \cap V(P))$. Since $G$ is $K_3$-free, $\Gamma(b) \cap V(P)$ and $\Gamma(c) \cap V(P)$ must be independent subsets of $V(P)$. But there is a unique way of covering $S$ by two independent subsets of $V(P)$ of size $3$, so without loss of generality, $\Gamma(b) \cap V(P)=\{x_2,y_3,z_3\}$ and $\Gamma(c) \cap V(P)=\{x_3,y_2,z_2\}$. Note that $b$ and $c$ are not adjacent to $x_1$, since $G$ is $K_3$-free. Finally, if $w \not \in \{b,c\}$, $w$ must be adjacent to $x_1$, $b$ or $c$, since $w$ is not adjacent to $v$ and $G$ is $K_3$-saturated. But $w$ is not adjacent to $b$ and $c$, since $d_=(w)=0$, so $w$ must be adjacent to $x_1$.  \\

Pick a vertex $a \in V(G) \setminus V(P)$ with $d_P(a)=1$. By Claim 5, without loss of generality, $\Gamma(a)=\{x_1,b,c\}$, for some vertices $b,c \in V(G) \setminus V(P)$ with $\Gamma(b) \cap V(P)=\{x_2,y_3,z_3\}$ and $\Gamma(c) \cap V(P)=\{x_3,y_2,z_2\}$. Then $G[V(P) \cup \{a,b,c\}]=Q$. \\

We now show that $\Gamma(v)=\Gamma(a)$ for every $v \in V(G) \setminus V(P)$ with $d_P(v)=1$. We first show that $\Gamma(v) \cap V(P)=\{x_1\}$. Indeed, if not, by Claim 5, without loss of generality, $\Gamma(v)=\{x_2,d,e\}$, for some vertices $d,e \in V(G) \setminus V(P)$ with $\Gamma(d) \cap V(P)=\{x_1,y_3,z_3\}$ and $\Gamma(e) \cap V(P)=\{x_3,y_1,z_1\}$. But then $\Gamma(v) \cap \Gamma(a) = \emptyset$, which is a contradiction, since $v$ and $a$ are not adjacent and $G$ is $K_3$-saturated. So $\Gamma(v) \cap V(P)=\{x_1\}$. By Claim 5 again, $\Gamma(v)=\{x_1,d,e\}$, for some vertices $d,e \in V(G) \setminus V(P)$, and for every $w \in V(G) \setminus V(P)$ with $d_P(w)=3$, $w \not \in \{d,e\}$ if and only if $w$ is adjacent to $x_1$ if and only if $w \not \in \{b,c\}$. Hence $\{d,e\}=\{b,c\}$, so $\Gamma(v)=\Gamma(a)$. \\ 

We now show that for every $v \in V(G) \setminus V(Q)$ with $d_P(v)=3$, we have $\Gamma(v) = \Gamma(w) \cap V(Q)$, for some vertex $w \in \{s,y_1,z_1\}$. Since $d_=(v)=0$ and $\Gamma(u) \subseteq V(Q)$ for all $u \in V(G) \setminus V(P)$ with $d_P(u)=1$, we have $\Gamma(v) \subseteq V(P)$. Since $G$ is $K_3$-free, $\Gamma(v)$ is an independent subset of $V(P)$. Since $G$ is $K_3$-saturated, $\Gamma(v)$ is in fact a maximally independent subset of $V(P)$. Moreover, by Claim 5, $x_1 \in \Gamma(v)$. But the only maximally independent subsets of $V(P)$ of size $3$ containing $x_1$ are the neighbourhoods in $Q$ of $s$, $y_1$ and $z_1$. \\

Finally, note that the set $\{s,y_1,z_1,a\}$ of vertices in $V(Q)$ of degree $3$ in $Q$ is an independent set. Hence $G$ is obtained from $Q$ by blowing up (non-adjacent) vertices of degree $3$.

\end{proof}

Recall (see section \ref{sec con}) that Conjecture \ref{con vert} is true for $t<2(r-2)$. In particular, Conjecture \ref{con vert} is true when $t=r \geq 5$. Recall also (see section \ref{sec con} again) that Conjecture \ref{con vert} is true when $t=r=4$. Hence $c(r,r)=\binom{r+1}{2}+9$ and $C(r,r)=\{P^{r-3},Q^{r-3}\}$ for all integers $r \geq 3$. \\

Finally, we show that $c'(3,2)=c(3,2)=5$ and $C'(3,2)=\{P\}$. So when $r=3$ and $t=2$ the answer to Question \ref{q} is again ``Yes.", but this time $C'(r,t)$ is not empty. Let $G$ be a $K_3$-saturated graph with $\delta(G) \geq 3$. We need to show that $e(G) \geq 2|G|-5$, with equality if and only if $G=P$. If $|G|<10$, we have $e(G) \geq 3|G|/2 > 2|G|-5$, so suppose $|G| \geq 10$. If $\delta(G) \geq 4$, we have $e(G) \geq 2|G|$, so suppose $\delta(G)=3$. Then by Theorem \ref{P,Q}, we have $e(G) \geq 3|G|-15 \geq 2|G|-5$, with equality if and only if $G$ is obtained from either $P$ or $Q$ by blowing up non-adjacent vertices of degree $3$ and $|G|=10$. But $|P|=10$ and $|Q|=13$, so we have equality if and only if $G=P$. \\

The known values of $c(r,t)$, $C(r,t)$, $c'(r,t)$ and $C'(r,t)$ are summarised in the table below. The values of $c'(r,r-1)$ and $C'(r,r-1)$ are only known for $r=3$ and the values of $c'(r,r)$ and $C'(r,r)$ are unknown.

\setlength{\tabcolsep}{18pt}
\renewcommand{\arraystretch}{1.5}

\begin{table}[h!]
\centering
\begin{tabu}{ |[0.5mm]c|[0.5mm]c|c|c|c|[0.5mm] }
 \tabucline[0.5mm]{-}
  & $c(r,t)$ & $C(r,t)$ & $c'(r,t)$ & $C'(r,t)$ \\ 
 \tabucline[0.5mm]{-}
 $t=r-2$ & $\binom{r-1}{2}$ & $\{K_{r-1}\}$ & $\binom{r-1}{2}$ & $\emptyset$ \\  
 \hline
 $t=r-1$ & $\binom{r}{2} + 2$ & $\{C_5^{r-3}\}$ & $c'(3,2)=5$ & $C'(3,2)=\{P\}$ \\  
 \hline
 $t=r$ & $\binom{r+1}{2} + 9$ & $\{P^{r-3},Q^{r-3}\}$ & & \\  
 \tabucline[0.5mm]{-}
\end{tabu}
\caption{The known values of $c(r,t)$, $C(r,t)$, $c'(r,t)$ and $C'(r,t)$.}
\end{table}

\section{Ideas for further work}\label{concl}

This section is organised as follows. In section \ref{exact} we speculate about the exact value of $c(3,t)$ for large $t$. In section \ref{remove} we discuss how one might remove the $\min{(r,\sqrt{t-r+3})}$ factor from the upper bounds in Theorems \ref{c est} and \ref{vert bound} .

\subsection{A speculative exact value of $c(3,t)$}\label{exact}

Recall our lower bound 

\begin{equation}\label{e300}
c(3,t) \geq t|G_t|-e(G_t) = \frac{1}{2} \binom{2(t-1)}{t-1} + 2t(t-1)
\end{equation}   

for $t \geq 3$ (see the proof of Theorem \ref{c est} in section \ref{mainresults}). Note that we have equality when $t=3$ and indeed $G_3=P$. By Theorem \ref{c est}, (\ref{e300}) is tight up to a constant factor. Moreover, the definition of $G_t$ seems fairly natural. Given all this, one might think that perhaps we have equality in (\ref{e300}), but this turns out to not be the case for $t \geq 6$. Indeed, for each integer $t \geq 4$, let $H_t$ be the graph obtained from $G_{t-1}$ by blowing up the vertices in $R$ (see section \ref{constr}) by two. Then it is easy to check that $H_t$ is $K_3$-saturated, $\delta(H_t)=t$, $|H_t|=2\binom{2(t-2)}{t-2} + 2(t-2)$ and $e(H_t)=2(t-1)\binom{2(t-2)}{t-2}$. Hence

\begin{equation}\label{e310}
c(3,t) \geq t|H_t|-e(H_t) = 2 \binom{2(t-2)}{t-2} + 2t(t-2)
\end{equation} 

for $t \geq 4$. For $t \geq 6$, the right hand side of (\ref{e310}) is strictly greater than that of (\ref{e300}), though only by a lower order term. Perhaps we have equality in (\ref{e310}) for large enough $t$ (and perhaps even for $t \geq 6$).

\begin{probl}
Determine the exact value of $c(3,t)$.
\end{probl}

\subsection{Ideas to remove the $\min{(r,\sqrt{t-r+3})}$ factor}\label{remove}
  
In this section we discuss how one might remove the $\min{(r,\sqrt{t-r+3})}$ factor from the upper bounds in Theorems \ref{c est} and \ref{vert bound} to match the lower bounds. In Theorems \ref{c est} and \ref{vert bound}, the factor arises from having to sum over all possible values of $c$ and $b$ in the proofs of Lemma \ref{tech} and Theorem \ref{vert bound}, respectively. More precisely, in the proofs of Lemma \ref{tech} and Theorem \ref{vert bound} we had the following.

\begin{definition}\label{var mod set}

For integers $a$, $b$ and $c$ with $a \geq c \geq -1$, an $(a,b,c)$ \emph{variable modified set system} is a sequence $(A_i,B_i,c_i)_{i \in I}$ of triples, where $A_i$ and $B_i$ are finite sets and $\max{(0,1-b)} \leq c_i \leq c$ is an integer, indexed by a finite, ordered set $I$, with the following properties.

\begin{enumerate}

    \item $|A_i \cap \cup_{k < i} \ A_k| \leq a$ and $|B_i| \leq b+2c_i$ for all $i \in I$.
    \item $|A_i \cap B_i| \leq c_i  $ for all $i \in I$.
    \item $|A_i \cap B_j| > c_j $ for all $i < j \in I$.
    
\end{enumerate}

\end{definition}

In the proof of Lemma \ref{tech}, $(A_i,E_i, \ r-3-w(D_i))_{i \in J_{a,b}}$ was a $(t-a, \ t-2r+6-b, \ \min{(t-a-1,r-3)})$ variable modified set system for all integers $1 \leq a,b \leq t$ by (\ref{e14}), (\ref{e15}), (\ref{e11}), (\ref{e12}), (\ref{e3}), (\ref{e16}) and the definition of $I_{a,b,c}$, where $J_{a,b} = \cup_{t-a > r-3-c} \ I_{a,b,c}$. In the proof of Theorem \ref{vert bound}, $(A_i,D_i, \ r-3-w(C_i))_{i \in (I_a \setminus J_a ) \setminus \{f\} }$ was a $(t-a, \ t-2r+6, \ \min{(t-a-1,r-3)})$ variable modified set system for all integers $1 \leq a \leq t $ by (\ref{e30}), (\ref{e27}), (\ref{e28}), (\ref{e29}), (\ref{e18}), (\ref{e31}) and the definition of $J_a$. \\

% $I_d=I_{k,l,r-3-d}$ for all $d$

%$I_d=I_{k,r-3-d}$ for all $d$, respectively.

In the proofs of Lemma \ref{tech} and Theorem \ref{vert bound} we bounded $|I|$ for these $(a,b,c)$ variable modified set systems as follows. For each integer $\max{(0,1-b)} \leq d \leq c$, let $I_d=\{i \in I : \ c_i=d\}$. Then the $I_d$ partition $I$ and $|I_d| \ll \binom{a+b+1}{a-d+1}$ for all $d$ by Theorem \ref{TF}, so $|I| \ll \sum_{\max{(0,1-b)} \leq d \leq c} \binom{a+b+1}{a-d+1}$. Recall that by partitioning $I$ into the $I_d$, we discarded the information that condition 3 in Definition \ref{var mod set} holds for $i$ and $j$ in different parts. One might think that if one could somehow bound $|I|$ directly, avoiding this partitioning, as we did for case 3 in the proof of Lemma \ref{tech} and for cases 2 and 3 in the proof of Theorem \ref{vert bound}, one would obtain a better bound for $|I|$, enabling one to remove the $\min{(r,\sqrt{t-r+3})}$ factor, but this turns out to not be the case. Indeed, we now give an example of an $(a,b,c)$ variable modified set system which shows that, even if one obtained a better bound for $|I|$, the $\min{(r,\sqrt{t-r+3})}$ factor would still arise. \\

Let $a$, $b$ and $c$ be integers with $a \geq c \geq -1$. Let $I$ be the set of all integers $\max{(0,1-b)} \leq d \leq c$, ordered with the dual of the usual order. For each $d \in I$, let $(A_i,B_i)_{i \in I_d}$ be the collection of all partitions of a set $S_d$ of size $a+b-1$ into subsets $A_i$ and $B_i$ of size $a-d$ and $b+d-1$, respectively. Let $T$ be an ordered set of size $c+1$ disjoint from all the $S_d$. For each integer $0 \leq e \leq c+1$, let $T_e$ be the set of the first $e$ elements of $T$. Take the $I_d$ to be disjoint and order them arbitrarily. Let $J=\cup_{d \in I} \ I_d$, ordered with the sum of the orders on the $I_d$. For each $d \in I$ and $i \in I_d$, let $A'_i=A_i \cup T_d$, $B'_i=B_i \cup T_{d+1}$ and $c_i=d$. Then it is easy to check that $(A'_i,B'_i,c_i)_{i \in J}$ is an $(a,b,c)$ variable modified set system with $|J|=\sum_{\max{(0,1-b)} \leq d \leq c} \binom{a+b-1}{a-d}$. \\

Note that this comes close to our upper bound. In particular, in the proof of Lemma \ref{tech}, (\ref{e14}), (\ref{e15}), (\ref{e11}), (\ref{e12}), (\ref{e3}) and (\ref{e16}) can be satisfied with $\left|\bigcup_{t-1>r-3-c} \ I_{1,1,c}\right|$ as large as 

$$ \sum_{\substack{0 \leq d \leq r-3 \\ 2r-t-4 \leq d}} \binom{2t-2r+3}{t-d-1} \ \asymp \ \frac{4^{t-r} \min{(r,\sqrt{t-r+3})}}{\sqrt{t-r+3}}  $$

(for $t>r-2$) and in the proof of Theorem \ref{vert bound}, (\ref{e30}), (\ref{e27}), (\ref{e28}), (\ref{e29}), (\ref{e18}) and (\ref{e31}) can be satisfied with $\left|(I_1 \setminus J_1) \setminus \{f\}\right|$ as large as 

$$ \sum_{\substack{0 \leq d \leq r-3 \\ 2r-t-5 \leq d \\ }} \binom{2t-2r+4}{t-d-1} \ \asymp \ \frac{4^{t-r} \min{(r,\sqrt{t-r+3})}}{\sqrt{t-r+3}} \ . $$

\leavevmode \newline So a different approach is required to remove the $\min{(r,\sqrt{t-r+3})}$ factor. In the proofs of Lemma \ref{tech} and Theorem \ref{vert bound} we had a sequence $(A_i,B_i)_{i \in I}$ of pairs of sets of vertices in a graph, indexed by a finite, ordered set $I$, such that $w(A_i \cap B_i) \leq r-3$ for all $i \in I$ and $w(A_i \cap B_j) > r-3$ for all $i<j \in I$ ((\ref{e3}) and (\ref{e4}) in the proof of Lemma \ref{tech} and (\ref{e18}) and (\ref{e19}) in the proof of Theorem \ref{vert bound}). We then applied Lemma \ref{Graph Theory Lemma} to $G[A_i \cap B_i]$ for each $i \in I$ to obtain a set $S_i \subseteq A_i \cap B_i$ such that $|S_i|-w(S_i)=|A_i \cap B_i|-w(A_i \cap B_i)$ and $|S_i| \geq 2w(S_i)$ and showed that we have $|A_i \cap B_i \setminus S_i| \leq r-3-w(S_i)$ for all $i \in I$ and $|A_i \cap B_j \setminus S_j| > r-3-w(S_j)$ for all $i<j \in I$ ((\ref{e8}), (\ref{e9}), (\ref{e11}) and (\ref{e12}) in the proof of Lemma \ref{tech} and (\ref{e25}), (\ref{e26}), (\ref{e28}) and (\ref{e29})  in the proof of Theorem \ref{vert bound})). But the same argument shows that in fact $|A_i \cap B_j \setminus S_k| > r-3-w(S_k)$ for all $i < j \in I$ and $k \in I$. Hence, in the proofs of Lemma \ref{tech} and Theorem \ref{vert bound} we had the following.

\begin{definition}\label{gen mod set}

For integers $a$, $b$, $c$ and $d$ with $a \geq d$, $b \geq c+d$ and $c,d \geq 0$, an $(a,b,c,d)$ \emph{generalised modified set system} is a sequence $(A_i,B_i,C_i)_{i \in I}$ of triples of finite sets with $C_i \subseteq A_i \cap B_i$, indexed by a finite, ordered set $I$, with the following properties.

\begin{enumerate}

    \item $|A_i \cap \cup_{k < i} \ A_k| \leq a$, $|B_i| \leq b$ and $|C_i| \geq c$ for all $i \in I$.
    \item $|A_i \cap B_i \setminus C_i| \leq d $ for all $i \in I$.
    \item $|A_i \cap B_j \setminus C_k| > d $ for all $i < j \in I$ and $k \in I$.
    
\end{enumerate}

\end{definition}

In the proof of Lemma \ref{tech}, $(A_i,B_i,D_i)_{i \in I_{a,b,c}}$ was a $(t-a, \ t-b, \ 2c, \ r-3-c)$ generalised modified set system for all integers $1 \leq a,b \leq t$ and $0 \leq c \leq \min{(r-3,t-r+2-b)}$ with $t-a > r-3-c$ by (\ref{e14}), the definition of $I_{a,b}$, (\ref{e9}), (\ref{e11}) and the generalisation of (\ref{e12}). In the proof of Theorem \ref{vert bound}, $(A_i,B_i,C_i)_{i \in I_{a,b}}$ was a $(t-a, \ t, \ 2b, \ r-3-b)$ generalised modified set system for all integers $1 \leq a \leq t$ and $\max{(0,a-t+r-2)} \leq b \leq \min{(r-3,t-r+2)}$ by the definition of $I_a$, (\ref{e23}), (\ref{e26}), (\ref{e28}) and the generalisation of (\ref{e29}). \\

In the proofs of Lemma \ref{tech} and Theorem \ref{vert bound} we bounded $|I|$ for these $(a,b,c,d)$ generalised modified set systems as follows. Taking $k=j$ in condition 3 in Definition \ref{gen mod set}, we obtain that $(A_i,B_i \setminus C_i)_{i \in I}$ is an $(a,b-c,d)$ modified set system, so $|I| \ll \binom{a+b-c-2d+1}{a-d+1}$ by Theorem \ref{TF}. The author believes that using condition 3 for all $k \in I$ rather than just $k=j$, one should be able to obtain a better bound for $|I|$, enabling one to remove the $\min{(r,\sqrt{t-r+3})}$ factor.

\begin{probl}
Estimate the maximum possible size of the index set in a generalised modified set system.
\end{probl}

\section*{Acknowledgement}

The author thanks Robert Johnson for his guidance and useful feedback on a draft of this paper.

\end{document}